\documentclass[11pt,reqno]{amsart}
\usepackage{amsmath,amssymb,amsthm,stmaryrd,latexsym,amsfonts,mathtools,tikz} 
\usepackage[top=30truemm,bottom=30truemm,left=25truemm,right=40truemm]{geometry}
\usepackage{mhequ}
\usepackage{comment}
\usepackage{cprotect}
\usepackage{verbatim}
\usepackage{orcidlink}
\usepackage{microtype}
\usepackage{mathrsfs}
\usepackage{enumitem}
\usepackage{hyperref}
\usepackage{xstring}

\usepackage{ulem}

\newtheorem{thm}{Theorem}[section]
\newtheorem{lemm}[thm]{Lemma}

\newtheorem{rem}[thm]{Remark}
\newtheorem{cor}[thm]{Corollary}

\newtheorem{assum}[thm]{Assumption}
\theoremstyle{definition}

\makeatletter
\@addtoreset{equation}{section}

\makeatother

\def\Wick#1{\mathord{{:}{#1}{:}}}
\def\scal#1{\langle #1\rangle}

\def\one{\mathbf{1}}

\newcommand{\eqdef}{\stackrel{\mbox{\tiny\rm def}}{=}}

\let\coloneqq\eqdef

\colorlet{RED}{red!90!black}
\colorlet{darkblue}{blue!90!black}
\colorlet{darkred}{red!90!black}

\def\CP{\mathcal{P}}
\def\D{\mathcal{D}}
\def\C{\mathbb{C}}
\def\Z{\mathbb{Z}}
\def\N{\mathbb{N}}
\def\E{\mathbb{E}}
\def\R{\mathbb{R}}
\def\P{\mathbb{P}}
\def\CX{\mathcal{X}}

\allowdisplaybreaks

\def\dash{\leavevmode\unskip\kern0.18em--\penalty\exhyphenpenalty\kern0.18em}
\def\slash{\leavevmode\unskip\kern0.15em/\penalty\exhyphenpenalty\kern0.15em}

\makeatletter

\DeclareRobustCommand{\TitleEquation}[2]{\texorpdfstring{\StrLeft{\f@series}{1}[\@firstchar]$\if%
		b\@firstchar\boldsymbol{#1}\else#1\fi$}{#2}}

\makeatother

\begin{document}

\title{Singularity of solutions to singular SPDEs}

\author{Martin Hairer\orcidlink{0000-0002-2141-6561}}
\address{EPFL, Lausanne, Switzerland and Imperial College London, UK}
\email{martin.hairer@epfl.ch}

\author{Seiichiro Kusuoka\orcidlink{0000-0001-6961-5036}}
\address{Graduate School of Science, Kyoto University, Kitashirakawa-Oiwakecho, Sakyo-ku, Kyoto 606-8502, Japan}
\email{kusuoka@math.kyoto-u.ac.jp}

\author{Hirotatsu Nagoji}
\address{Graduate School of Science, Kyoto University, Kitashirakawa-Oiwakecho, Sakyo-ku, Kyoto 606-8502, Japan}
\email{nagoji.hirotatsu.63x@st.kyoto-u.ac.jp}

\begin{abstract} 
Building on the notes \cite{HaiOrthogonal}, we give a sufficient condition for the marginal distribution of the solution of singular SPDEs on the $d$-dimensional torus ${\mathbb T}^d$ to be singular with respect to the law of the Gaussian measure induced by the linearised equation.
As applications we obtain the singularity of the $\Phi ^4_3$-measure with respect to the Gaussian free field measure and the border of parameters for the fractional $\Phi ^4$-measure to be singular with respect to the Gaussian free field measure.
Our approach is applicable to quite a large class of singular SPDEs.
\end{abstract}

\maketitle
\tableofcontents
\thispagestyle{empty}

\section{Introduction}

We give a sufficient\dash and ``almost'' necessary, see Remark~\ref{rem:trickGirsanov}\dash condition for the marginal 
distribution of the solution to nonlinear singular SPDEs on the $d$-dimensional torus ${\mathbb T}^d$ to be singular with respect 
to the (Gaussian) law of the corresponding free evolution.
Our condition is applicable to the $\Phi ^4_3$-quantum field model and also the fractional $\Phi ^4$-quantum field models.

We consider singular SPDEs of the form
\begin{equation}\label{koj}
\mathcal{L} u + F(u) = \xi \quad \mbox{on}\ \mathbb{R}_+ \times \mathbb{T}^d = \mathbb{R}_+ \times (\mathbb{R}/2\pi\mathbb{Z})^d,
\end{equation}
where $\mathcal{L}$ is a parabolic operator and $\xi$ is a Gaussian random noise.
In this paper, we consider the case when $\mathcal{L} = \partial_t + \scal{\nabla}^\sigma$, $\xi = \langle \nabla \rangle ^m \eta$ and $F(u) =\langle \nabla \rangle^{n_0}\prod_{i=1}^k  \left(\langle \nabla \rangle^{n_i} u\right)$ with $n_i\in\mathbb{R}$, where $\langle \nabla \rangle = (1-\Delta )^{\frac{1}{2}}$ and $\eta$ denotes Gaussian space-time white noise.
This class of equations is slightly different from those that are usually considered in this context since it 
involves $\langle \nabla \rangle$ rather than $\nabla$, so our equation is non-local unless these exponents are all 
even integers. Because of the difference, the KPZ equation is not included in our setting.
For technical reasons we replace $\nabla$ by $\langle \nabla \rangle$,
see Remark~\ref{rem:KPZcase} for more details.

When $\mathcal{L} = \partial_t + 1 - \Delta$, $F(u)= \lambda u^3$ with $\lambda>0$, and $\xi$ is space-time white noise, this equation is called the dynamical $\Phi^4_d$ model (on the torus), which 
is a singular SPDE that is subcritical in the sense of regularity structures when $d \in [2,4)$. The associated invariant measure 
is the classical $\Phi^4_d$ model which is formally given by
\begin{equ}[e:Phi4]
\mu(d\phi) \propto e^{-\frac{\lambda}{4} \int_{\mathbb{T}^d} \phi^4(x) dx} \mu_0(d\phi), 
\end{equ}
where $\mu_0$ denotes Nelson's massive Gaussian free field,  i.e.\ the Gaussian measure on distributions over
$\mathbb{T}^d$ with covariance $(1-\Delta)^{-1}$.
It is known that the measure $\mu$ (when interpreted with suitable renormalisation) is absolutely continuous with respect to $\mu_0$ when $d=2$ (see e.g. \cite{AR91}), but is singular when $d=3$ (see \cite{HaiOrthogonal, BG21}).
In this article, we study a more general case, namely we compare the marginal distribution of the solution 
to \eqref{koj} (for any fixed time $t>0$ and any initial condition)
with the reference Gaussian measure given by the solution to the linear evolution $\mathcal{L}Z = \xi$, and 
we give a sufficient condition for their mutual singularity.

Loosely speaking, the main result of the present paper, Theorem~\ref{new}, suggests 
the following rule of thumb, where one should think of $t>0$ as fixed:
\begin{equation}\label{eq:loose_statement}\begin{array}{c}
\mbox{\it ``If $u(t)-Z(t)$ does not belong to the Cameron--Martin space of the law of $Z(t)$}\\
\mbox{\it almost surely, then the laws of $u(t)$ and $Z(t)$ are mutually singular.''}
\end{array}\end{equation}
Here $u$ is a solution to the (suitably renormalized) equation \eqref{koj} in the sense of singular SPDEs.
We stress that in the decomposition $u(t) = Z(t) + v(t)$ which typically appears in the 
solution theory for singular SPDEs, $Z(t)$ and $v(t)$ are far from being independent, so that 
 \eqref{eq:loose_statement} does not follow from the Cameron--Martin theorem. In 
fact, as we will see in Remark~\ref{rem:KPZcase} below, there are situation where~\eqref{eq:loose_statement} \textit{fails} due to the  
presence of additional cancellations.
One example is that of the KPZ equation where we know \cite{Jeremy} that the invariant measure coincides
with that of the linearised equation. Since that measure is essentially Wiener measure, its
Cameron--Martin space is given by the Sobolev space $H^1$. However, it is possible to show
in this particular example that while the process $u-Z$ does belong to $H^s$ for every $s < 1$,
it almost surely does \textit{not} belong to $H^1$. 
Another example where the heuristic fails is that of the 2D stochastic Navier--Stokes equations driven
by noise with power law spectrum, whose invariant measure was shown to be equivalent to that of the corresponding
Ornstein--Uhlenbeck process \cite{LeonardoJames}. 
The reason why the above heuristics fails in these cases is 
that our proof relies crucially
on the logarithmic divergence of the constant $c^\alpha_{N,2}$ given by \eqref{e:cN2} with $\alpha = \alpha_0$ as in 
\eqref{e:defalpha0}, but this constant happens to be bounded in these cases. This is due to the 
geometric structure of the nonlinearity which ``almost'' preserves the Cameron--Martin norm of the invariant measure.

\begin{rem}\label{rem:trickGirsanov}
The converse of the above rule of thumb, namely that $u(t)-Z(t)$ belonging to the Cameron-Martin 
space of the law of $Z(t)$ implies that the laws of $u(t)$ and $Z(t)$ are mutually equivalent
was shown to hold for a relatively wide class of parabolic SPDEs in \cite{Girsanov,Girsanov2,MRS22}.
The idea developed in these articles is that even though one cannot apply
Girsanov's theorem at fixed time, one can do so in path space. When doing this naively one 
obtains much more stringent criteria for absolute continuity. The trick is to play with Duhamel's
formula in such a way that one constructs an auxiliary process $\hat u^{(t)}$ which is ``closer''
to $Z$ on path space than $u$ and is such that, for the fixed target time $t$, one has
$\hat u^{(t)}(t) = u(t)$ (but $\hat u^{(t)}(s) \neq u(s)$ for $s \in [0,t)$). 
\end{rem}

As an application of our main theorem, we can see that the $\Phi^4_3$-measure on the torus, which appears in the Euclidean quantum filed theory, is singular with respect to the corresponding Gaussian free 
field (see Section~\ref{sec:appPhi4}). Note that this has already been proven in \cite{HaiOrthogonal, BG21} and, in fact, it had been ``known'' to be the case for much longer as a consequence of the 
construction in \cite{Glimm} which showed that the $\Phi^4_3$ Hamiltonian acts on a 
physical Hilbert space of states that is different from the Fock space associated to the GFF.
Our result is also applicable to the fractional $\Phi ^4$-model, and in particular we confirm the 
previously conjectured boundary
of parameters for the fractional $\Phi ^4$-measure to be singular (see Section~\ref{sec:appfrPhi4}).
  
The absolute continuity and singularity of the fractional $\Phi ^4$-measure have already been studied in \cite{liu2024}.
The method applied there relies on the variational representation of the measure, which is very different from 
the method in the present paper.
Indeed, our argument is based on the decomposition typically appearing in the solution theory of singular SPDEs (see Assumption~\ref{katei}) and the failure of the remainder term to belong to the Cameron--Martin space 
(see \eqref{eq:loose_statement} and Remark~\ref{rem:CM}).
We emphasize that the results in \cite{liu2024} and in the present paper are of course consistent with each other (see Corollary~\ref{cor:fPhi4} and Remark~\ref{rem:fPhi4}).

A related model, namely the Hartree-type $\Phi ^4$-measures, was also studied using the variational 
representation of the measures (see \cite[Theorem~1.5]{Bjoern1} and \cite[Theorem~1.12]{OOT}).
One important aspect is that, contrary to these results, the method of proof used here also applies to invariant measures
for singular stochastic PDEs that are not of gradient type, so that no explicit representation of the measure
exists.


The organization of the present paper is as follows.
In Section~\ref{sec:notations} we give the notations which are used in the present paper, and in Section~\ref{se2} we give the main theorems and their proofs.
In the proofs in Section~\ref{se2} we skip explicit calculations by referring lemmas in Section~\ref{sele} and postpone them to Section~\ref{sele}.
In Section~\ref{appli} we give applications of the main results.
The most interesting applications are the $\Phi ^4_3$-model and the fractional $\Phi ^4$-models.
Precisely, we obtain the singularity of the $\Phi ^4_3$-measure obtained by the method of singular SPDEs with respect to the Gaussian free field measure, and also the border of parameters for the fractional $\Phi ^4$-measure to be singular (see Section~\ref{sec:appfrPhi4}).
Unfortunately, our main theorem is not applicable to the KPZ equation.
However, it is applicable to a similar but different equation to the KPZ equation (see Section~\ref{sec:appKPZ}).
In Section~\ref{sele} we prove the lemmas which referred in Section~\ref{se2} with explicit but complicated calculations.

\section{Notations}\label{sec:notations}

We introduce some notations which are used in this paper.
\begin{itemize}
\item
We write $a \lesssim b$ if $a \le Cb$ holds for some constant $C>0$ which is independent of the variables under consideration. We write $a\simeq b$ if one has $a \lesssim b$ and $a \gtrsim b$.
\item We use the conventions $\N = \{1,2,\ldots\}$ and $\N_0 = \N \cup \{0\}$.
We also use the shorthand $\Z^d_N = \{k \in \Z^d\,:\, |k| \le N\}$.
\item
We define the Fourier transform $\hat f$ of a smooth function $f \in \D \coloneqq C^\infty( \mathbb{T}^d)$ by
\begin{equation}
\hat f (l) \coloneqq \langle e_l, f \rangle = (2\pi )^{-\frac{d}{2}} \int_{\mathbb{T}^d} f(x) e^{-i l\cdot x} dx, \ \ \ l\in\mathbb{Z}^d,
\end{equation}
where $\left\{ e_l \right\}_{l\in\mathbb{Z}^d}$ is defined by $e_l (x) \coloneqq (2\pi )^{-\frac{d}{2}} e^{il\cdot x}$ and $\langle \cdot  ,\cdot \rangle$ denotes the inner product on $L^2(\mathbb{T}^d;\mathbb{C})$.
\item
Let $\D'$ be the space of distributions on $\mathbb{T}^d$, i.e.\ the topological dual of $\D$. Then, we extend the above definition of the Fourier transform to $f \in \D'$ in the usual way. We also extend the $L^2$-inner product $\langle \cdot , \cdot \rangle$ to the duality between $\D'$ and $\D$.
\item
For any function $a:\C^d\rightarrow \C$ of at most polynomial growth, we define the Fourier multiplier $a(\nabla)$ by
\begin{equation}
a(\nabla)f \coloneqq \sum_{l \in \Z^d} a (il)  \hat f (l) e_l,\ \ \ f\in \D'\;.
\end{equation}
In the special case when we can write $a(x) = b(|x_1|^2 + \cdots + |x_d|^2)$ for some $b \colon \R \to \C$, we also write $a(\nabla) = b(-\Delta)$. 
\item
We define $\langle \cdot \rangle: \C^d \rightarrow \mathbb{R}_+$ by 
\begin{equation}
\langle x \rangle \coloneqq (|x|^2 + 1)^{\frac{1}{2}}.
\end{equation}
\item
Let $W^{\alpha, p}(\mathbb{T}^d),\ \alpha\in\mathbb{R},\ 1\le p\le \infty$ be the Sobolev space, which is defined by the completion of $\D$ under the norm
\[   \| f \|_{W^{\alpha, p}(\mathbb{T}^d)} \coloneqq \|  \langle \nabla \rangle^\alpha f \|_{L^p(\mathbb{T}^d)} =  \|  (1-\Delta)^{\frac{\alpha}{2}} f \|_{L^p(\mathbb{T}^d)} .\]
When $p=2$, we write $H^\alpha (\mathbb{T}^d) \coloneqq W^{\alpha,2}(\mathbb{T}^d)$.
\item
Let $\mathcal{C}^\alpha (\mathbb{T}^d),\ \alpha\in\mathbb{R}$ be the H\"older-Besov space, see Section~\ref{kpp} for the definition.
\item
We define the approximation operators $P_N,\ N\in\mathbb{N}_0$ by
\begin{equation}
P_N f \coloneqq \sum_{l\in\mathbb{Z}^d,\ |l|\le N} \hat f(l) e_l .
\end{equation}
\item
For $k\in\mathbb{N}_0$ and $c\in\mathbb{R}$, we define the $k$th Hermite polynomial $H_k (x;c)$ by
imposing that $H_k$ is of degree $k$ and that
\begin{equation}\label{herm}
e^{tx-\frac{1}{2}ct^2} = \sum_{k=0}^\infty \frac{t^k}{k!} H_k(x;c),\ \ \ t,x\in\mathbb{R}. 
\end{equation}
We extend this notation to the multicomponent case in the usual way  by considering multi-indices 
$\beta\in\mathbb{N}_0^k$ and setting
\begin{equation}\label{polyw}
e^{\langle \boldsymbol{t}, \boldsymbol{x} \rangle_{\mathbb{R}^k}-\frac{1}{2}\langle C \boldsymbol{t}, \boldsymbol{t}\rangle_{\mathbb{R}^k}} = \sum_{\beta\in\mathbb{N}_0^k} \frac{\boldsymbol{t}^\beta}{\beta!} H_\beta(\boldsymbol{x};C),\ \ \ \boldsymbol{t},\boldsymbol{x}\in\mathbb{R}^k, 
\end{equation}
where $C=(C_{ij})_{1\le i,j\le k} \in \mathbb{R}^k \otimes \mathbb{R}^k$ is a symmetric matrix. These polynomials are 
orthogonal for the Gaussian distribution with covariance $C$ and are closely related to the Wiener chaos decomposition, 
see Section~\ref{piyo}.
\end{itemize}

\section{Main result and its proof}\label{se2}
In this section, we state and prove our main result. We prove some lemmas which are used in the proof later in Section~\ref{sele}.
Letting $\sigma \ge 0$ and $F(u) =\langle \nabla \rangle^{n_0}\prod_{i=1}^k  \left(\langle \nabla \rangle^{n_i} u\right)$ with $n_i\in\mathbb{R}$, we consider the equation
\begin{equ}\label{koll} 
 \partial_t u + (1 - \Delta)^{\frac{\sigma}{2}} u + F(u) =  \langle \nabla \rangle^{m} \eta\;,\qquad \mbox{on}\ \mathbb{R}_+ \times \mathbb{T}^d\;,
\end{equ}
where $\eta$ is Gaussian space-time white noise defined on a probability space $(\Omega,\mathcal{F},\mathbb{P})$.  Then, the stationary solution $Z$ of the linearised equation
\begin{gather}\label{koji} 
 \partial_t Z + (1 - \Delta)^{\frac{\sigma}{2}} Z =  \langle \nabla \rangle^{m} \eta, 
\end{gather}
satisfies the following, see Section~\ref{sele} for the proof.
\begin{lemm}\label{nimai}
One has  $Z(t) \in\mathcal{C}^{\frac{\sigma - d}{2} - m -\kappa}(\mathbb{T}^d)$ $\mathbb{P}$-almost surely for any $t, \kappa >0$.
\end{lemm}
Throughout the paper, we make the following assumption to simplify matters, although we do not
expect this assumption to be strictly necessary for our results, see Remark~\ref{hirq} below.
\begin{assum}\label{w43}
One has $\frac{\sigma -d}{2}- m -n_i <0$ for any $1 \le i \le k$.
\end{assum}
Loosely speaking, this assumption guarantees that none of the factors appearing in $F$ are function-valued,
so that $F(Z)$, and thus $F(u)$ is classically ill-defined. Under a certain subcriticality
condition however, see \cite[Assumption~8.3]{Hai14}, it is known that we can give a meaning to them via 
renormalization. By the theory of regularity structures \cite{Hai14}, we can expect at least the existence of local (in time) solutions to the equation as long as it is subcritical. Now, letting $u$ be the solution of \eqref{koll} after suitable renormalization, we consider the shifted process $v(t) \coloneqq u(t) - Z(t) + Y(t)$, where the process $Y$ is defined as follows. 
Writing $Z_N = P_N Z$, we define
\[ \Wick{F(Z_N)} = \scal{\nabla}^{n_0} H_{(1,1,\cdots,1)} \left( N(\nabla) Z_N ; C_N    \right),\]
where $N(\nabla) \coloneqq (\langle \nabla \rangle^{n_1}, \cdots,\langle \nabla \rangle^{n_k})$, $C_{N;ij} \coloneqq \mathbb{E} \left[ \left(\langle \nabla \rangle^{n_i}Z_N (t)\right) \left(\langle \nabla \rangle^{n_j}Z_N(t) \right) \right]$ (which is independent of $t$), and the polynomial $H_{(1,1,\cdots,1)}$ is defined in Section~\ref{sec:notations}.
Then, we define
\[Y^N(t) \coloneqq \int^t_{-\infty}e^{(s-t)(1-\Delta)^{\frac{\sigma}{2}}}(\Wick{F(Z_N)})(s) ds ,\]
which is the stationary solution of 
\begin{equation}\label{hiku}
\mathcal{L}Y^N = \Wick{F(Z_N)}\;,\qquad \mathcal{L} \coloneqq \partial_t + (1 - \Delta)^{\frac{\sigma}{2}} \;,   
\end{equation}
where $(e^{-t(1-\Delta)^{\frac{\sigma}{2}}})_t$ is the semigroup generated by the fractional Laplacian $(1-\Delta)^{\frac{\sigma}{2}}$.
Let $A$ be the exponent given by
\begin{equ}[e:defA]
A \coloneqq \frac{1}{2} \sum_{i=1}^k(\sigma - d - 2m -2n_i)\;,
\end{equ}
which represents the regularity of the nonlinearity before the application of $\scal{\nabla}^{n_0}$. 
In particular one has $A < 0$ by Assumption~\ref{w43}.
We then make the following assumption.
\begin{assum}\label{ass:subcritical}
One has 
\begin{equ}[e:subcritical]
A + \frac{d+\sigma}{2} > 0\;,\qquad
A + \frac{d+\sigma}{2} > n_0 - m\;.
\end{equ}
\end{assum}
The second condition in \eqref{e:subcritical} is nothing but the assumption that the equation
under consideration is subcritcial in the sense of \cite{Hai14}. The first condition
rules out the kind of ``variance blow-up'' phenomenon observed for example in \cite{variance}.
We prove the following lemma in Section~\ref{sele}.

\begin{lemm}\label{hima}
Under Assumptions~\ref{w43}--\ref{ass:subcritical}, the sequence of random variables $Y^N(t)$ converges to some $Y(t)$ $\mathbb{P}$-almost surely in $\mathcal{C}^{A - n_0 + \sigma -\kappa}(\mathbb{T}^d)$ and in $L^p(\Omega; \mathcal{C}^{A - n_0 + \sigma -\kappa}(\mathbb{T}^d))$ as $N\rightarrow \infty$ for any $t, \kappa, p>0$.
\end{lemm}
Under the subcriticality assumption, we can expect the spatial regularity of $v$ to be strictly better than 
that of $Y$, see \cite{Hai14,BCCH21}. This kind of expansion is common when dealing with singular SPDEs and is often referred to 
as the Da Prato--Debussche trick, after \cite{DPD02,DPD03}.
This observation justifies the following assumption which we take for granted in lieu of a solution theory 
for \eqref{koll}.

\begin{assum}\label{katei}
There exists a (nonrandom) time $t>0$ such that, setting $v \coloneqq u - Z + Y$, one has $v(t)\in \mathcal{C}^{\beta}(\mathbb{T}^d)$ almost surely for some $\beta > A -n_0 + \sigma$.
\end{assum}

\begin{rem}\label{zaq}
Although in the above argument we have not discussed the existence of solutions satisfying Assumption~\ref{katei}, it turns
out that in many cases of interest, equations of the type \eqref{koll} with suitable renormalization 
are known to be globally well-posed and satisfy Assumption~\ref{katei}. 
For example, it is proved that the dynamical $\Phi^4_3$-model is one of them.
\end{rem}

\begin{rem}
The effect of the initial condition $u(0)$ on the solution $u(t)$ appears as the term $e^{-t(1-\Delta)^{\frac{\sigma}{2}}} u(0)$ in view of the mild form of the solution, and it belongs to $\mathcal{C}^{\beta}(\mathbb{T}^d)$ for any $\beta >0$.
Hence, we regard it included in the term $v(t)$.
\end{rem}

\begin{rem}\label{hirq}
In this paper, we exclude the case where Assumption~\ref{w43} fails, in which case $v$ typically fails to be more
regular than $Y$. Although our 
proof is not directly applicable then, we expect that\dash at least in some situations\dash such equations can also 
be treated with similar techniques.
\end{rem}

Now we state the first main theorem.
In the theorem we only assume that $u(t)$ satisfies Assumption~\ref{katei} with $t>0$, and do not assume that $u$ is a solution of \eqref{koll}, because in the proof we only use the structure of $u(t)$ appeared in Assumption~\ref{katei}.
Here we remark the information of \eqref{koll} is included in $Z(t)$ and $Y(t)$, which appear in Assumption~\ref{katei}.

\begin{thm}\label{new}
Let $u$ be any process satisfying Assumption~\ref{katei} at some time $t>0$, let Assumptions~\ref{w43}--\ref{ass:subcritical} hold, and let
\begin{equ}[e:singular]
A + \frac{\sigma}{2} \le n_0-m\;.
\end{equ}
Then, the law of $u(t)$ is singular with respect to the law of $Z(t)$.
\end{thm}

\begin{rem}\label{rem:CM}
Set $r_1 \coloneqq \frac{\sigma - d}{2} -m$ and $r_2 \coloneqq A -n_0 + \sigma$. In view of Lemma~\ref{hima}, these can be regarded as the spatial regularities of $Z$ and $Y$, respectively. Then, the condition $A + \frac{\sigma}{2} \le -m + n_0$ is equivalent to $r_2 - r_1 \le \frac{d}{2}$. The case $r_2-r_1 = \frac{d}{2}$ is the border of the condition that $Y(t)$ belong to the Cameron--Martin space of the law of $Z(t)$. 
Note that the critical case $A + \frac{\sigma}{2} = -m + n_0$ is included in Theorem~\ref{new}.
\end{rem}


Before we proceed to the proof, we introduce some notations and lemmas. First, we define for any $\alpha \in \R$ 
the sequences of constants 
\begin{equation}
c_{N,1}^\alpha \coloneqq \E \left[ \left| \langle \nabla \rangle ^\alpha Z_N (t,x)   \right|^2   \right]
\end{equation}
(which is indeed independent of $x$ and $t$ by space-time stationarity of $Z$) 
and
\begin{equ}[e:cN2]
c_{N,2}^\alpha  \coloneqq \E \left[ H_{k}\left(  \langle \nabla \rangle ^\alpha Z_N (t,x);c_{N,1}^\alpha  \right) \langle \nabla \rangle ^\alpha Y_N (t,x)  \right]\;,
\end{equ}
where $Y_N \coloneqq P_N Y$ and $H_k$ is the $k$th Hermite polynomial, see Section~\ref{sec:notations} for the definition.
Recall that $k$ is the order of nonlinearity, as fixed at the start of Section~\ref{se2}. Then, we consider the events
\begin{equs}
A^{\alpha,\gamma} &\coloneqq \Bigl\{ \phi\in \D' \,:\, 
 \lim_{N \to \infty} N^{-\gamma} \langle H_{k+1}\left( \langle \nabla \rangle ^\alpha    P_N\phi; c_{N,1}^\alpha   \right) + (k+1)c_{N,2}^\alpha, 1   \rangle_{L^2(\mathbb{T}^d)} = 0  \Bigr\} \;,\\
B^{\alpha,\gamma} &\coloneqq \Bigl\{ \phi\in  \D' \,:\,
\lim_{N \to \infty} (\log N)^{-\gamma} \langle H_{k+1}\left( \langle \nabla \rangle ^\alpha    P_N\phi; c_{N,1}^\alpha   \right) + (k+1)c_{N,2}^\alpha, 1   \rangle_{L^2(\mathbb{T}^d)} = 0  \Bigr\}\;,
\end{equs}
for $\alpha, \gamma \in \mathbb{R}$. Finally, we prepare the following lemmas, see Section~\ref{sele} for the proofs.

\begin{lemm}\label{anb}
\begin{enumerate}[label=(\roman*)]
\item\label{(i)}
When $\gamma >0$ and $J\in \mathbb{N}$ satisfy $J(\frac{\sigma - d}{2}- m - \alpha) + \gamma >  - \frac{d}{2}$, the sequence $\left\{ N^{-\gamma}\Wick{(\langle \nabla \rangle^{\alpha} Z_N)^J}(t) \right\}_N$ converges to $0$ almost surely in $\mathcal{C}^{-\frac{d}{2}}(\mathbb{T}^d)$ and in $L^p(\Omega;\mathcal{C}^{-\frac{d}{2}}(\mathbb{T}^d) )$ as $N\rightarrow \infty$ for any $t, p >0$.
\item\label{(ii)}
Assume that $J\in \mathbb{N}$ is such that $J(\frac{\sigma - d}{2}- m - \alpha) =  - \frac{d}{2}$.
Then, for $\gamma > \frac{1}{2}$, the sequence 
$\left\{(\log N)^{-\gamma}\Wick{(\langle \nabla \rangle^{\alpha} Z_N)^J}(t)\right\}_N$ converges to $0$ almost surely in $\mathcal{C}^{-\frac{d}{2}}(\mathbb{T}^d)$ and in $L^p(\Omega;\mathcal{C}^{-\frac{d}{2}}(\mathbb{T}^d) )$ as $N\rightarrow \infty$ for any $t, p>0$.
\end{enumerate}
\end{lemm}
\begin{rem}
From Lemmas~\ref{nimai} and~\ref{vv}, one would expect that $\{ N^{-\gamma}(\langle \nabla \rangle^\alpha Z_N (t))^J \}_N$ converges to $0$ only when $J\left( \frac{\sigma - d}{2} - m- \alpha \right) + \gamma >0$. Lemma~\ref{anb} shows that, thanks to renormalisation, $\{ N^{-\gamma}\Wick{\langle \nabla \rangle^\alpha Z_N )^J}(t) \}_N$ converges to $0$ under much weaker conditions on $\gamma$.
\end{rem}

A special role is being played by the value
\begin{equ}[e:defalpha0]
\alpha_0 = \frac{\sigma-d}{2} - m +\frac{1}{k+1}\left(A + m - n_0 + \frac{\sigma+d}{2}   \right)\;,
\end{equ}
which is such that $\delta(\alpha_0) = 0$, where
\begin{equs}
\delta(\alpha) &\coloneqq k\left(\alpha + m - \frac{\sigma - d}{2} \right) - \left( A -n_0 + \sigma - \alpha   \right) \\
&= n_0 + \alpha - \sigma + \sum_{i=1}^k \bigl(2m+n_i + \alpha + d-\sigma\bigr) \;.
\end{equs}
From Lemmas~\ref{nimai} and~\ref{hima}, $-\delta(\alpha)$ can be formally regarded as the sum of the spatial regularities of $\Wick{(\langle \nabla \rangle^\alpha Z (t))^k}$ and $\langle \nabla \rangle^\alpha Y(t)$, which means that $\alpha_0$ is the borderline value of $\alpha$ for the product $\Wick{(\langle \nabla \rangle^\alpha Z (t))^k} \langle \nabla \rangle^\alpha Y(t) $ to be well-defined. This suggests that the expectation $c_{N,2}^\alpha$ of this product may diverge whenever
$\alpha \ge \alpha_0$. This is indeed the case for our choice of nonlinearity, as shown in the following lemma.

\begin{lemm}\label{lpj}
Under Assumption~\ref{w43}--\ref{ass:subcritical}, one has 
\begin{equ}[e:boundc2]
 c_{N,2}^\alpha \gtrsim  
\begin{cases}\notag
N^{\delta(\alpha)} \quad  \mbox{when}\ \delta(\alpha) >0, \\
 \log N \quad  \mbox{when}\  \delta(\alpha) =0.
\end{cases}
\end{equ}
Furthermore, there exists $\alpha_\star > \alpha_0$ such that, 
for any $\alpha > \alpha_\star$, one has
\begin{equ}[e:boundZY]
\sup_N N^{\frac d2 - \delta(\alpha)} \E \bigl|\bigl\langle\Wick{\left(  \langle \nabla \rangle^\alpha Z_N  \right)^{k}}(t)\langle \nabla \rangle ^\alpha Y_N (t) - c_{N,2}^\alpha(t), 1 \bigr\rangle_{L^2(\mathbb{T}^d)}\bigr| < \infty\;.
\end{equ}
When $A + \frac{\sigma}{2} = n_0-m$ and $\alpha = \alpha_0$, then one has
\begin{equ}[e:boundZYborderline]
\sup_N \E \bigl|\bigl\langle\Wick{\left(  \langle \nabla \rangle^\alpha Z_N  \right)^{k}}(t)\langle \nabla \rangle ^\alpha Y_N (t) - c_{N,2}^\alpha(t), 1 \bigr\rangle_{L^2(\mathbb{T}^d)}\bigr| < \infty\;.
\end{equ}
\end{lemm}

\begin{rem}
The prefactor $N^{\frac d2 - \delta(\alpha)}$ appearing in \eqref{e:boundZY} can be understood as follows.
As long as $\alpha$ is not too large, one would expect the term $\Wick{\left(\langle \nabla \rangle^\alpha Z_N  \right)^{k}}\langle \nabla \rangle ^\alpha Y_N - c_{N,2}^\alpha$ to converge to a non-trivial finite limit as $N \to \infty$. 
The regularity of this limit is given by $-\delta(\alpha)$ as already mentioned before the lemma.
It is however well known that this must break down at the latest when the regularity reaches $-\frac d2$, namely the regularity
of white noise, see for example \cite{CQ,variance} for an illustration of this phenomenon. 
Since this term contains $(k+1)\alpha$ derivatives, this suggests that the expectation appearing in \eqref{e:boundZY} is
at least of order $N^{(k+1)\alpha + \kappa}$, where $\kappa$ is such that $(k+1)\alpha + \kappa = 0$ when $\delta(\alpha) = \frac d2$.
Our result then states that this is in fact sharp \dash in particular there are no stronger subdivergences that haven't been taken
care of by our renormalisation.
\end{rem}

Now, we prove Theorem~\ref{new}.

\begin{proof}[Proof of Theorem~\ref{new}]
We first consider the case when the inequality \eqref{e:singular} is strict, which in particular
implies that
\begin{equ}[e:inequalityalpha]
(k+1) \Bigl(\frac{\sigma-d}2-m-\alpha\Bigr) + \delta(\alpha) > -\frac{d}2\;.
\end{equ}

Fix some $\alpha > \alpha_\star$ with $\alpha_\star$ as in Lemma~\ref{lpj}. Note that we have $\delta (\alpha) > 0$, since $\delta(\alpha) > \delta (\alpha_\star) > \delta(\alpha_0) =0$ holds by $\alpha > \alpha_\star > \alpha_0$. Thus, by choosing $\gamma < \delta(\alpha)$ sufficiently close to $\delta(\alpha)$, we have $\gamma>0$, $\gamma > \delta(\alpha) - \frac{d}{2}$ and by \eqref{e:inequalityalpha}, $(k+1)(\frac{\sigma -d }{2}-m-\alpha) + \gamma > - \frac{d}{2}$, which is the condition appearing in Lemma~\ref{anb}.    

Since $\alpha > \alpha_0$, it also follows from the second inequality in \eqref{e:subcritical} that one has 
\begin{align*}
0= \delta(\alpha_0) < \delta(\alpha) = k\left( \alpha + m - \frac{\sigma -d}{2}  \right) - (A-n_0 + \sigma - \alpha) < (k+1)\left( \alpha + m - \frac{\sigma - d}{2} \right)
\end{align*} 
and thus $\frac{\sigma - d}{2} - m - \alpha <0$. In particular, if condition~\ref{(i)} of Lemma~\ref{anb} is satisfied for $J=k+1$, then it is necessarily also satisfied for smaller values of $J$.
Writing $u_N \coloneqq P_N u$ and $v_N \coloneqq P_N v$, we have
\begin{align}
 &N^{-\gamma} H_{k+1}\left( \langle \nabla \rangle^\alpha u_N ; c_{N,1}^\alpha     \right) \notag \\
 & = N^{-\gamma} H_{k+1}\left( \langle \nabla \rangle^\alpha \left( Z_N + v_N - Y_N   \right);  c_{N,1}^\alpha     \right) \notag \\
 &= N^{-\gamma}\sum_{i=0}^{k+1} \binom{k+1}{i}\Wick{\left(  \langle \nabla \rangle^\alpha Z_N  \right)^{i}}   \left( \langle \nabla \rangle^\alpha \left(  v_N - Y_N   \right)\right)^{k+1-i} \notag \\
 &= N^{-\gamma} \sum_{i=0}^{k+1} \sum_{j=0}^{k+1-i} \binom{k+1}{i} \binom{k+1-i}{j}\Wick{\left(  \langle \nabla \rangle^\alpha Z_N  \right)^{i}} \left( - \langle \nabla \rangle^\alpha  Y_N \right)^{j} \left( \langle \nabla \rangle^\alpha   v_N \right)^{k+1-i-j} \notag \\
 &= N^{-\gamma}\Wick{\left(  \langle \nabla \rangle^\alpha Z_N  \right)^{k+1}} - N^{-\gamma}(k+1)\Wick{\left(  \langle \nabla \rangle^\alpha Z_N  \right)^{k}} \left(  \langle \nabla \rangle^\alpha  Y_N \right) \notag \\
 &\  + N^{-\gamma} \!\!\!\!\!\!\!\!\!\!\!\!\!\!\! \sum_{\substack{0\le i+j \le k+1 \\ (i,j)\neq (k+1,0),(k,1)}}  \binom{k+1}{i} \binom{k+1-i}{j}\Wick{\left(  \langle \nabla \rangle^\alpha Z_N  \right)^{i}} \left( - \langle \nabla \rangle^\alpha  Y_N \right)^{j} \left( \langle \nabla \rangle^\alpha   v_N \right)^{k+1-i-j}\;. \label{cx}
\end{align}
Our choices of $\alpha$ and $\gamma$ guarantee that the first term in the right-hand side of \eqref{cx} converges to $0$ in $\D'$ $\mathbb{P}$-almost surely as $N\rightarrow \infty$ in view of Lemma~\ref{anb}.
For the third term, we have 
\begin{align*}
&\Wick{\left(\langle \nabla \rangle^\alpha Z_N  \right)^i} \left(  \langle \nabla \rangle^\alpha Y_N   \right)^j  \left(  \langle \nabla \rangle^\alpha v_N   \right)^{k+1-i-j}\\
&= H_i\left(\langle \nabla \rangle^\alpha Z_N ; c_{N,1}^\alpha    \right)\left(  \langle \nabla \rangle^\alpha Y_N   \right)^j  \left(  \langle \nabla \rangle^\alpha v_N   \right)^{k+1-i-j}\\
&=\sum_{p=0}^{[\frac{i}{2}]} C_{i,i-2p} (c_{N,1}^\alpha)^p \left( \langle \nabla \rangle^\alpha Z_N \right)^{i-2p}\left(  \langle \nabla \rangle^\alpha Y_N   \right)^j  \left(  \langle \nabla \rangle^\alpha v_N   \right)^{k+1-i-j},  
\end{align*}
where $C_{i,i-2p}$ are the coefficients of the Hermite polynomial. Therefore, since $c_{N,1}^\alpha \simeq N^{2\alpha + 2m - (\sigma - d)}$, in order to prove that the third term in \eqref{cx} converges to $0$ as $N\rightarrow \infty$, it is enough to check that for each $i,j,p$,
\begin{align*}
N^{-\gamma + \left(\alpha + m - \frac{\sigma - d}{2}\right)2p}\left( \langle \nabla \rangle^\alpha Z_N \right)^{i-2p}\left(  \langle \nabla \rangle^\alpha Y_N   \right)^j  \left(  \langle \nabla \rangle^\alpha v_N   \right)^{k+1-i-j} \rightarrow 0
\end{align*}
as $N\rightarrow \infty$. In view of Lemmas~\ref{nimai},~\ref{hima},~\ref{vv} and Assumption~\ref{katei}, it is enough to prove that
\begin{align}
(i -2p + 2p)\left( \frac{\sigma - d}{2} - m- \alpha \right) + j(A -n_0 + \sigma - \alpha) + (k + 1 - i - j)(\beta - \alpha) + \gamma >0.  \label{abc}
\end{align}
We have from Assumption~\ref{katei} that 
\begin{align*}
&(i -2p + 2p)\left( \frac{\sigma - d}{2} - m- \alpha \right) + j(A -n_0 + \sigma - \alpha) + (k + 1 - i - j)(\beta - \alpha) \\
&= i \left( \frac{\sigma - d}{2} - m- \alpha \right) + j(A -n_0 + \sigma - \alpha) + (k + 1 - i - j)(\beta - \alpha)\\
&\begin{cases}
      > i \left( \frac{\sigma - d}{2} - m- \alpha \right) + (k+ 1 -i)(A -n_0 + \sigma - \alpha) \ge -\delta(\alpha) \ \ \ \mbox{if}\  i+j<k+1 \\
      = i \left( \frac{\sigma - d}{2} - m- \alpha \right) + (k+ 1 -i)(A -n_0 + \sigma - \alpha) > - \delta(\alpha) \ \ \ \mbox{if}\ i+j= k+1,
         \end{cases} 
\end{align*}
where in the last inequality, we used the condition $(i,j)\neq (k+1,0),(k,1)$ and the inequality $ \frac{\sigma - d}{2} - m- \alpha <  A -n_0 + \sigma - \alpha$, which follows from the second inequality in \eqref{e:subcritical}. 
Therefore, provided that $\gamma$ is close enough to $\delta(\alpha)$, we indeed have \eqref{abc} and thus the third term in \eqref{cx} converges to $0$ in $\mathcal{D}'$ $\mathbb{P}$-almost surely as $N\rightarrow \infty$.
Thus, from Lemma~\ref{lpj}, we get ${\mathbb P}(u(t)\in A^{\alpha, \gamma}) = 1$. On the other hand, from $N^{-\gamma}c_{N,2}^\alpha \rightarrow \infty$ and $ N^{-\gamma}\Wick{\left(  \langle \nabla \rangle^\alpha Z_N (t)  \right)^{k+1}} \rightarrow 0$ as $N\rightarrow \infty$ $\mathbb{P}$-almost surely, which follow from Lemmas~\ref{anb} and~\ref{lpj}, we get ${\mathbb P}(Z(t)\in A^{\alpha, \gamma}) = 0$. This proves the mutual singularity of the laws of $Z(t)$ and $u(t)$.

We now turn to the borderline case when one has equality in \eqref{e:singular}.
In this case, we set $\alpha = \alpha_0$ and we choose any $\gamma\in (\frac{1}{2}, 1)$.
Similarly to above, we get
\begin{align}
&(\log N)^{-\gamma} H_{k+1}\left( \langle \nabla \rangle^\alpha u_N ; c_{N,1}^\alpha     \right) \notag \\
&= (\log N)^{-\gamma}\Wick{\left(  \langle \nabla \rangle^\alpha Z_N  \right)^{k+1}} - (\log N)^{-\gamma}(k+1)\Wick{\left(  \langle \nabla \rangle^\alpha Z_N  \right)^{k}} \left(  \langle \nabla \rangle^\alpha  Y_N \right) \label{cxx}\\
 &\quad + (\log N)^{-\gamma}\sum_{\substack{0\le i+j \le k+1 \\ (i,j)\neq (k+1,0),(k,1)}}  \binom{k+1}{i} \binom{k+1-i}{j} \notag \\
&\quad \hspace{4cm} \times\Wick{\left(  \langle \nabla \rangle^\alpha Z_N  \right)^{i}} \left( - \langle \nabla \rangle^\alpha  Y_N \right)^{j} \left( \langle \nabla \rangle^\alpha   v_N \right)^{k+1-i-j}.  \notag 
 \end{align}
Then, from Lemmas~\ref{anb} and~\ref{pro}, the first and third terms in the right-hand side of \eqref{cxx} converges to 0 in $\D'$ $\mathbb{P}$-almost surely as $N\rightarrow \infty$. Therefore, from \eqref{e:boundZYborderline}, we get $\mathbb{P}( u(t)\in B^{\alpha, \gamma}) = 1$. On the other hand, from Lemmas~\ref{anb} and~\ref{lpj}, we get $(\log N)^{-\gamma}c_{N,2}^\alpha \rightarrow \infty$ and $ (\log N)^{-\gamma}\Wick{\left(  \langle \nabla \rangle^\alpha Z_N (t)  \right)^{k+1}} \rightarrow 0$ as $N\rightarrow \infty$ $\mathbb{P}$-almost surely. This means $\mathbb{P}( Z(t)\in B^{\alpha, \gamma}) = 0$.
This proves the orthogonality of the law of $u(t)$ with respect to the law of $Z(t)$.
\end{proof}

\begin{rem}\label{rem:KPZcase}
We only considered the case that $N_i(\nabla) = \langle \nabla \rangle^{n_i}$ for $i=0,1,\ldots,k$ and $M(\nabla) = \langle \nabla \rangle^m$, but our argument is applicable to the case of more 
general Fourier multipliers.
For example, when $N_i$ and $M$ satisfy $|N_i(l)|_{\mathbb{C}} \lesssim |l|^{n_i}$, $|M(l)|_{\mathbb{C}} \lesssim |l|^{m}$ as $|l| \rightarrow \infty$, and the sequence $\{ c_{N,2}^\alpha\}_N$ diverges exactly as stated in Lemma~\ref{lpj}, it is straightforward to extend our result to this case. 
From the proof of Lemma~\ref{lpj}, the value of $c_{N,2}^\alpha$ is given by
\begin{equation}\label{teisu}
c_{N,2}^\alpha = k! \sum_{l_i \in \mathbb{Z}^d_N} \frac{\langle l_1 + \cdots  + l_k \rangle^{\alpha}N_0(l_1 + \cdots + l_k)}{\sum_{i=1}^k \langle l_i \rangle^\sigma +\langle l_1 + \cdots l_l \rangle^\sigma }\prod_{i=1}^k 
\langle l_i \rangle^{\alpha -\sigma} N_i(l_i) |M(l_i)|^2 .
\end{equation}
The divergence of $\{ c_{N,2}^\alpha\}_N$ is crucial in the proof of Theorem~\ref{new}, and our argument is not applicable if it fails to diverge.
This is the case for some equations, for example, the KPZ equation and the stochastic Burgers equation. In our setting, the (one-dimensional) KPZ equation is the case that $d=1$, $\sigma =2$, $k=2$, $N_0(l) =-1$, $N_1(l)=N_2(l)= i l$, and $M(l)=1$ for $l\in\mathbb{Z}$. Inserting these into \eqref{teisu}, we see that $c_{N,2}^\alpha =0$ from the antisymmetry of $N_1$  and $N_2$, which is consistent with the fact that 
the conclusion of Theorem~\ref{new} fails there. The case of stochastic Burgers equation is similar.
See Section~\ref{sec:appKPZ}.
\end{rem}
The first part of Remark~\ref{rem:KPZcase} can be summarized as follows.
\begin{thm}\label{thm:generalcase}
Consider the setting of Theorem~\ref{new}, but in the definitions of $Z$ and $Y^N$, we replace 
$\scal{\nabla}^m \eta$ by $M(i\nabla)\eta$ and we set
\begin{equ}
F(u) = N_0(i\nabla) \prod_{i=1}^k (N_i(i\nabla)u)\;,
\end{equ}
where $M$ and $N_i$ are continuous functions $\R^d \to \R$ such that
\begin{equation}\label{growthMN}
|M(p)| \lesssim |p|^{m}\;,\qquad
|N_i(p)| \lesssim |p|^{n_i}\;,
\end{equation} 
as $|p|\rightarrow \infty$. 
Then, the conclusion of Theorem~\ref{new} still holds, under the additional assumption that the sequence of constants $c_{N,2}^\alpha$ defined by \eqref{e:cN2} satisfies the condition \eqref{e:boundc2}.
\end{thm}

\begin{proof}
We proceed almost exactly the same way as in the proof of Theorem~\ref{new}.
More precisely, we show that for the same $\alpha, \gamma$ as in the proof of Theorem~\ref{new}, we have
$\mathbb{P}(u(t)\in A^{\alpha,\gamma} )=1$ and $\mathbb{P}(Z(t) \in A^{\alpha,\gamma})=0$ ($\mathbb{P}(u(t)\in B^{\alpha,\gamma} )=1$ and $\mathbb{P}(Z(t) \in B^{\alpha,\gamma})=0$ in the borderline case). Note that Lemmas~\ref{nimai}, \ref{hima}, \ref{anb}, \eqref{e:boundZY} and \eqref{e:boundZYborderline} still hold true in the setting of Theorem~\ref{thm:generalcase} since their proofs essentially depend only on the \textit{upper} bounds \eqref{growthMN} for the Fourier multipliers. For example, the proofs of Lemmas~\ref{hima} and \ref{anb} are reduced to how we control the sums \eqref{sumY} and \eqref{sum:Z:}, respectively.
The corresponding quantities in the current setting are given by
\[ \sum_{\tau \in \mathfrak{S}_k} \sum_{l_i \in \mathbb{Z}^d_N} \langle l_1 + \cdots + l_k \rangle^{2\alpha} \frac{\langle l_1 + \cdots + l_k \rangle^{-\sigma}}{\sum_{i=1}^k \langle l_i \rangle^\sigma + \langle l_1 + \cdots + l_k \rangle^{\sigma}}\prod_{i=1}^k \langle l_i \rangle^{-\sigma} |M(l)|^2 |N_i(l)| |N_{\tau(i)}(l)|   \]
and
\[ \sum_{l_i \in \mathbb{Z}^d_N} \langle l_1 + \cdots + l_{k+1} \rangle^{2\beta} \prod_{i=1}^J \langle l_i \rangle^{-\sigma+ 2\alpha} |M(l)|^2 .   \]
Since these are dominated by \eqref{sumY} and \eqref{sum:Z:} by the condition \eqref{growthMN}, we 
conclude that the conclusions of Lemmas~\ref{hima} and \ref{anb} still hold true in this broader setting. 
The same applies to Lemma~\ref{nimai}, \eqref{e:boundZY} and \eqref{e:boundZYborderline}. 
 Consequently, the fact that $\mathbb{P}(u(t)\in A^{\alpha,\gamma} )=1$ follows in the exact same way 
 as in the proof of Theorem~\ref{new}. For the proof that $\mathbb{P}(Z(t)\in A^{\alpha,\gamma} )=0$, we need to show that
\begin{equ}[e:wantedlimsup]
\limsup_{N\rightarrow \infty} N^{-\gamma} \Bigl|\left\langle  \Wick{(\langle \nabla \rangle^{\alpha} Z_N(t))^k} + (k+1)c_{N,2}^\alpha, 1 \right\rangle_{L^2(\mathbb{T}^d)}\Bigr| > 0\;,
\end{equ}
(or the analogous statement with $N$ replaced by $\log N$ in the borderline case) almost surely. 
We proved this in the case $N_i(p)= \langle p \rangle^{n_i}$ and $M(p) = \langle p \rangle^m$ by combining Lemma~\ref{anb} and the divergence of $c_{N,2}^\alpha$ as in \eqref{e:boundc2}. 

Unfortunately, our proof of \eqref{e:boundc2} depends on the fact that $N_i(p) \simeq |p|^{n_i}$, which
is the only point in the proof that a \textit{lower} bound is used and assuming the condition 
$|N_i(p)| \simeq |p|^{n_i}$ alone is not sufficient. Indeed, \eqref{e:boundc2} can typically fail to hold when 
 $N_i$ is antisymmetric as explained in Remark~\ref{rem:KPZcase} and as recently exploited 
 in \cite{LeonardoJames}. If however we take \eqref{e:boundc2} as an additional assumption, 
then \eqref{e:wantedlimsup} follows immediately from Lemma~\ref{anb} since the term inside the
lim sup is the sum of a divergent term and a term that converges to $0$.
\end{proof}

\section{Applications}\label{appli}

In this section, we see some applications of our main results.

\subsection{The \TitleEquation{\Phi ^4}{Phi^4}-model}\label{sec:appPhi4}

In dimensions $2$ and $3$, it is known \cite{MW17,Pavlos,Konstantin} that the classical $\Phi^4_d$ model \eqref{e:Phi4}
on the torus arises as the unique invariant measure 
for the Markov process associated to the (formal) SPDE:
\begin{equation}\label{eq:AppPhi4eq}
\partial _t u + (1- \Delta ) u + \lambda u^3 = \sqrt{2}\xi , \quad \mbox{on}\ \mathbb{R}_+\times \mathbb{T}^d
\end{equation} 
where $\xi$ is a Gaussian space-time white noise.
The cube appearing in this equation needs to be interpreted with an appropriate renormalization as
in \cite{Hai14}.
For more details on the $\Phi^4$-model and its stochastic quantization, we refer to the introductions of \cite{AK20, GH21, AK22}.
We remark that \eqref{eq:AppPhi4eq} is a special case of \eqref{koll} with $k=3$ and $m=n_0=\cdots =n_3=0$.

The advantage of Theorem~\ref{new} is that in the case that $d=3$ for any global-in-time solution $u(t)$ of \eqref{eq:AppPhi4eq} in the sense of of singular SPDEs, we can prove the singularity of the law of $u(t)$ ($t>0$) with respect to $\mu_0$ without discussing the uniqueness of the $\Phi ^4_3$-measure on ${\mathbb T}^d$, as follows.

\begin{cor}\label{cor:Phi4}
Let $\{ u(t)\,:\, t\in [0,T]\}$ be any global-in-time solution of the dynamic $\Phi^4_3$-model constructed 
in \cite{Hai14}.
Then, for any $t>0$ the law of $u(t)$ is singular with respect to the Gaussian free field measure.
\end{cor}

\begin{proof}
When we study \eqref{eq:AppPhi4eq} with $d=3$ by the singular SPDE method (including both the regularity structure and the paracontrolled calculus), one obtains the decomposition 
\begin{equation}\label{eq:corPhi4-01}
u(t) = Z(t) + \lambda Y(t) + v(t)
\end{equation}
with $u(t) \in  {\mathcal C}^\beta ({\mathbb T}^3)$ for all $\beta <-\frac{1}{2}$ and $v(t) \in {\mathcal C}^\beta ({\mathbb T}^3)$ for all $\beta <1$.
In particular, Assumption~\ref{katei} is satisfied.
Hence, we see that Theorem~\ref{new} is applicable and the assertion follows.
\end{proof}

\begin{rem}
As in the proof of Corollary~\ref{cor:Phi4}, to apply Theorem~\ref{new} we only use the fact that $u(t)$ is decomposed as in Assumption~\ref{katei}, which is the common technique in singular SPDEs.
This is the reason why Corollary~\ref{cor:Phi4} holds for any global-in-time solution constructed by the singular SPDE method.
In particular, the marginal distributions at time $t>0$ of the global-in-time solutions obtained in \cite{MW17, GH19, AK20, JP23} are singular with respect to the Gaussian free field measure, and so is the $\Phi ^4_3$-measure constructed in \cite{AK20}.
\end{rem}

\begin{rem}
In \cite[Section 2.8.2]{BCCH21} the $\Phi ^4_{4-\delta}$-model is considered.
This model is given by \eqref{eq:AppPhi4eq} with replacing $\xi$ by another Gaussian noise whose 
(lack of) regularity matches that of a hypothetical space-time white noise ``in dimension ${4-\delta}$''.
The $\Phi ^4_{4-\delta}$-model is also subcritical in the sense of regularity structures and 
it was shown in \cite{CMW23} that it admits global-in-time solutions.
For $\delta \in (0,1)$, by applying Theorem~\ref{new} one obtains the singularity of the law of 
$u(t)$ with respect to the Gaussian free field measure.
\end{rem}

\begin{rem}
A closely related problem is the study of $\Phi^4$ measures with a Hartree-type nonlinearity, which (at least formally) arises
as invariant measure for the SPDE
\begin{equ}[e:Hartree]
\partial _t u + (1- \Delta ) u + \lambda u \scal{\nabla}^{-\beta}(u^2) = \sqrt{2}\xi , \quad \mbox{on}\ \mathbb{R}_+\times \mathbb{T}^d\;.
\end{equ}
While this doesn't quite fit the form of the nonlinearity considered in \eqref{koj}, it is clear both from \eqref{e:singular}
and from our proof that the relevant quantity is the total number of derivatives $\sum_{i=0}^k n_i$ and not the
way in which they are distributed across powers of $u$. The condition \eqref{e:singular} then suggests that the 
invariant measure for the renormalised version of \eqref{e:Hartree} is singular with respect to the corresponding Gaussian 
measure if and only if $\beta \le \frac{3d-8}{2}$, which is consistent with the condition $\beta \le \frac12$
in dimensions $3$ which arose in the construction of \cite{Bjoern1,Bjoern2,OOT}.
\end{rem}

\subsection{The fractional \TitleEquation{\Phi^4}{Phi^4}-model}\label{sec:appfrPhi4}

Here, we consider the fractional $\Phi^4$-model, which is treated for example in \cite{GH21, Duc22, Duc23, liu2024, DGR24}
and is formally given by
\begin{equation}\label{paru}
\mu _{\Phi ^4_{d, \sigma}} (d\phi) = \frac{1}{Z}e^{-\lambda \int_{\mathbb{T}^d} \phi^4(x) dx}\mu_0^\sigma (d\phi),
\end{equation}
where $\mu_0^\sigma$ is the Gaussian measure on $\D'$ with covariance $(1-\Delta)^{-\frac{\sigma}{2}}$, and $\lambda >0$.
We remark that the covariance of $\mu_0^\sigma$ is sometimes given by $(1+(-\Delta)^{\frac{\sigma}{2}})^{-1}$ in the literature, but the replacement by $(1-\Delta)^{-\frac{\sigma}{2}}$ does not make much of a difference.
The parabolic stochastic quantization equation of $\mu _{\Phi ^4_{d, \sigma}}$ is given by
\begin{equation}\label{sq}
\partial_t u +  (1 - \Delta)^{\frac{\sigma}{2}}  u + \lambda u^{ 3} = \sqrt{2} \xi \quad \mbox{on}\ \mathbb{R}_+\times \mathbb{T}^d,
\end{equation}
where $\xi$ is Gaussian space time white noise on $\mathbb{R}_+\times \mathbb{T}^d$.
In the case $\sigma =2$, we recover the usual $\Phi ^4$ model.

When $\sigma \leq d$, we need to introduce proper renormalization for \eqref{paru} and \eqref{sq} to make sense, and the condition $\sigma>\frac{1}{2}d$ corresponds to the subcritical regime.
Similarly to the $\Phi ^4_3$-measure, despite its formal expression \eqref{paru}, $\mu _{\Phi ^4_{d, \sigma}}$ is not always absolutely continuous with respect to $\mu_0^\sigma$.
When $\frac{3}{4}d<\sigma\le d$, the $4$th Wick power $\Wick{\phi^4}$ is well-defined as a distribution-valued random variable $\mu_0^\sigma$-almost everywhere. Because it is also known that $e^{-\lambda \int _{\mathbb{T}^d}\Wick{\phi^4} (x)dx}\in L^p(\mu_0^\sigma)$ for any $p\ge 1$ in this case, $\mu _{\Phi ^4_{d, \sigma}}$ is absolutely continuous with respect to $\mu_0^\sigma$ (see e.g. \cite{Nag23}). 
On the other hand, when $\frac{1}{2}d<\sigma\le \frac{3}{4}d$, we need further renormalization beyond Wick ordering and it is expected that the measure $\mu$ is singular with respect to $\mu_0$ if it exists.
We can formally see this transition from our result: in the setting of this paper, the equation \eqref{sq} 
corresponds to the case $k=3$, $m=n_0=\cdots =n_3=0$, and hence $A = \frac{3}{2}(\sigma -d)$. Thus, the condition $-\frac{d}{2} <  A + \frac{\sigma}{2} \le -m + n_0$ in Theorem~\ref{new} is equivalent to $\frac{1}{2}d < \sigma \le \frac{3}{4}d$.
We summarize the argument here as a corollary.

\begin{cor}\label{cor:fPhi4}
For $\frac{3}{4}d<\sigma\le d$, the fractional $\Phi ^4$-measure $\mu _{\Phi ^4_{d, \sigma}}$ is absolutely continuous with respect to $\mu_0^\sigma$.
For $\frac{1}{2}d<\sigma\le \frac{3}{4}d$, $\mu _{\Phi ^4_{d, \sigma}}$ is singular with respect to $\mu_0^\sigma$  if it is constructed as a stationary measure of \eqref{sq}.
\end{cor}

\begin{rem}\label{rem:fPhi4}
This is consistent with the results of \cite{Barashkov,liu2024}
where the authors show that the fractional $\Phi^4$ measure in $3$ dimensions is singular with respect
to the free field measure when $\sigma \le 9/4$.
\end{rem}


\subsection{An equation similar to the KPZ equation}\label{sec:appKPZ}

The KPZ equation is also a typical model treated in the field of singular SPDEs, and is formally given by
\begin{equation}\label{eq:AppKPZ}
\partial _t h + (1- \partial _x^2 ) h - \left| \partial _x h\right| ^2 = \xi , \quad \mbox{on}\ \mathbb{R}_+\times \mathbb{T}
\end{equation} 
where $\xi$ is Gaussian space-time white noise on $\mathbb{R}_+\times \mathbb{T}$.
Because of the expected singularity of $h$, the term $\left| \partial _x h\right| ^2$ in \eqref{eq:AppKPZ} needs renormalization, see \cite{Hai13, GP17} for details on the background and solution theory for the KPZ equation.

Unfortunately, we cannot apply our results to \eqref{eq:AppKPZ} (see Remark~\ref{rem:KPZcase}) and in fact
the conclusion fails to hold in this case.
Instead of \eqref{eq:AppKPZ}, we consider the following equation, which is similar but different from \eqref{eq:AppKPZ}
\begin{equation}\label{eq:AppKPZ2}
\partial _t h + (1- \partial _x^2 ) h - \left| \langle \partial _x \rangle h\right| ^2 = \xi , \quad \mbox{on}\ \mathbb{R}_+\times \mathbb{T}
\end{equation}
where $\langle \partial _x \rangle := (1+ \partial _x ^2 )^{\frac{1}{2}}$ and $\xi$ is Gaussian space time white noise on $\mathbb{R}_+\times \mathbb{T}$.
We remark that \eqref{eq:AppKPZ2} is a special case of \eqref{koll}, in particular the case that $k=2$ and $m=n_0=0$ and $n_1=n_2=1$.
The nonlinear term $\left| \langle \partial _x \rangle h\right| ^2$ generates the same singularity as $\left| \partial _x h\right| ^2$ in \eqref{eq:AppKPZ}, so that one expects to be able to obtain 
a similar local solution theory (see \cite{Hai14} for the approach by regularity structures and \cite{GIP15, GP17} for the approach by paracontrolled calculus). This suggests that one can write
\begin{equation}\label{eq:corKPZ2-01}
h(t) = e^{-t(1-\Delta)}h(0) + Z(t) + I(\Wick{(\langle \partial _x \rangle Z)^2})(t) + v(t)
\end{equation}
with $h(t) \in  {\mathcal C}^\beta ({\mathbb T}^3)$ for all $\beta <\frac{1}{2}$ and $v(t) \in {\mathcal C}^\beta ({\mathbb T}^3)$ for all ${\beta <\frac{3}{2}}$.

By this observation, if the local-in-time solution $h$ of \eqref{eq:AppKPZ2} can be extended globally in time, Assumption~\ref{katei} is satisfied and we can conclude from Theorem~\ref{new} that for $t>0$ the law of $h(t)$ is singular with respect to the law of a Brownian bridge.
Unfortunately, the existence of the global-in-time solution of \eqref{eq:AppKPZ2} is not known, while it has been proved in the case of \eqref{eq:AppKPZ} (see \cite[Remark 1.5]{Hai14} or \cite[Section 7]{GP17}).

\section{Proofs of lemmas}\label{sele}

In this section, we prove the lemmas stated in Section~\ref{se2}.
Gaussian space-time white noise $\eta$ is an isometric operator $\eta: L^2(\mathbb{R}\times \mathbb{T}^d) \rightarrow L^2(\Omega,\mathcal{F},\mathbb{P})$ such that the range is included in a Gaussian space $H$. Then, $\{ \hat W (s,l)\}_{l\in\mathbb{Z}^d} = \{ \langle \eta , \one_{[0,s]} e_l \rangle\}_{l\in\mathbb{Z}^d}$ defines a family of $\mathbb{C}$-valued Brownian motions that are i.i.d.\ save for the fact that $\hat W (t,l)$ coincides with the complex conjugate of $\hat W (t,-l)$ and $\mathbb{E}[|\hat W (t,l)|^2]= t$ for any $l\in\mathbb{Z}^d$.
With these notations, the Fourier coefficients of the stationary solution $Z(t)$ of the equation \eqref{koji} are given by
\begin{equation}\label{kaik}
\hat Z (t,l) =  \int^t_{-\infty} e^{(s-t)\langle l \rangle^{\sigma}} \langle l \rangle^m   d \hat W (s,l)
\end{equation}
for $l\in\mathbb{Z}^d$. As an immediate consequence, we have the following expression.

\begin{lemm}\label{bu}
For $N,M\in \mathbb{N}, t_1,t_2\in\mathbb{R}$ and $x_1,x_2 \in \mathbb{T}^d$, there holds
\[ \mathbb{E} \left[Z_M(t_1,x_1)Z_N(t_2,x_2)    \right] = \sum_{l\in \mathbb{Z}^d, |l|\le M \wedge N} e^{- |t_1 -t_2|\langle l \rangle^\sigma}   \langle l \rangle^{2m-\sigma} e_l(x_1)e_{-l}(x_2). \]
\end{lemm}

We use the notations introduced in Section~\ref{piyo} in the following proofs.
First, we prove Lemma~\ref{hima}. We omit the proof of Lemma~\ref{nimai} because it can be proven by a similar argument to the proof of Lemma~\ref{hima}.

\begin{proof}[Proof of Lemma~\ref{hima}]
We consider the case $n_0 = 0$ since the general case then follows easily.
By the hypercontractivity of Gaussian polynomials
\begin{align}
\mathbb{E}\left[ \| Y^N (t) \|_{W^{\alpha, 2p}(\mathbb{T}^d)}^{2p}  \right] &= \mathbb{E}\left[ \| \langle \nabla \rangle^{\alpha} Y^N (t) \|_{L^{ 2p}(\mathbb{T}^d)}^{2p}  \right]\notag\\
&\lesssim \int_{\mathbb{T}^d} \mathbb{E} \left[ \left| \langle \nabla \rangle^{\alpha} Y^N (t,x)   \right|^2   \right]^p dx.
\end{align}
Moreover, from an explicit calculation similar to \cite[Theorem 4.3]{CC18} for example,
\begin{align}
&\mathbb{E} \left[ \left| \langle \nabla \rangle^\alpha  Y^N (t,x)      \right|^2   \right] \notag\\
&\lesssim \sum_{\tau \in \mathfrak{S}_k} \sum_{l_i \in \mathbb{Z}^d_N} \langle l_1 + \cdots + l_k \rangle^{2\alpha} \frac{\langle l_1 + \cdots + l_k \rangle^{-\sigma}}{\sum_{i=1}^k \langle l_i \rangle^\sigma + \langle l_1 + \cdots + l_k \rangle^{\sigma}}\prod_{i=1}^k \langle l_i \rangle^{-\sigma + 2m + n_i + n_{\tau(i)}}. \label{sumY}
\end{align}
It follows that 
$\sup_N \mathbb{E}\left[ \|  Y^N(t)      \|_{W^{\alpha, 2p}(\mathbb{T}^d)}^{2p}  \right] < \infty$,
provided that 
\begin{equation}\label{eq:Lem3.3-01}
\sum_{i=1}^k \left\{  (\sigma - n_i - n_{\tau(i)} - 2m) \wedge d \right\} + \sigma > \left\{(k-1)d \right\}\vee (2\alpha + kd)
\end{equation}
for any $\tau \in \mathfrak{S}_k$. 
Note that the inequality~\eqref{eq:Lem3.3-01} indeed holds from our assumption.
In view of the embedding $W^{\alpha,p}(\mathbb{T}^d) \subset W^{\alpha - \frac{d}{p},\infty}(\mathbb{T}^d) \subset C^{\alpha- \frac{d}{p}}(\mathbb{T}^d)$, by taking sufficiently large $p$, the desired result follows.
The $\mathbb{P}$-almost sure convergence follows from a standard argument using the Borel--Cantelli lemma, see \cite[Proposition 2.7]{OOT24} or \cite[Proposition 4.3]{Nag24} for example.
\end{proof}

Next, we prove Lemma~\ref{anb}.

\begin{proof}[Proof of Lemma~\ref{anb}]
Similarly to the proof of Lemma~\ref{hima}, we get from Corollary~\ref{contraction}
\begin{align}
\mathbb{E}\left[ \|\Wick{\left( \langle \nabla \rangle^\alpha Z_N  \right)^J} (t) \|_{W^{\beta, 2p}(\mathbb{T}^d)}^{2p}  \right] &\lesssim \int_{\mathbb{T}^d} \mathbb{E} \left[ \left| \langle \nabla \rangle^{\beta}\Wick{\left( \langle \nabla \rangle^\alpha Z_N  \right)^J}   (t,x)   \right|^2   \right]^p dx \notag \\
&\lesssim \biggl(\sum_{l_i \in \mathbb{Z}^d_N} \langle l_1 + \cdots + l_{k+1} \rangle^{2\beta} \prod_{i=1}^J \langle l_i \rangle^{-\sigma + 2m + 2\alpha}     \biggr)^p \label{sum:Z:}\\
&\lesssim 1+ N^{\left\{J(2\alpha + 2m - \sigma) + (2\beta)\vee (-d)    \right\}p + Jdp}\;,\notag
\end{align}
which proves assertion~\ref{(i)}. When $J\left( \frac{\sigma -d}{2} - m -\alpha  \right) = - \frac{d}{2}$,  we similarly obtain
\[ \mathbb{E}\left[ \|\Wick{\left( \langle \nabla \rangle^\alpha Z_N  \right)^J} (t) \|_{W^{-\frac{d}{2}, 2p}(\mathbb{T}^d)}^{2p}  \right] \lesssim (\log N)^p\;,\]
which proves assertion~\ref{(ii)}.  
\end{proof}

Next, we prove Lemma~\ref{lpj}.

\begin{proof}[Proof of Lemma~\ref{lpj}]
We define $G(f) \coloneqq \langle \nabla \rangle^{-n_0} F(f) = \prod_{i=1}^k\langle \nabla \rangle^{n_i}f$.
From Lemma~\ref{bu}, we get
\begin{align*}
c_{N,2}^\alpha &= \mathbb{E} \left[\Wick{\left( \langle \nabla \rangle ^\alpha Z_N   \right)^{k}}\langle \nabla \rangle ^\alpha Y_N (t,x)  \right]\\
&= \langle \nabla \rangle_{x'}^{\alpha + n_0} \int^t_{-\infty} ds\ e^{(s-t) \langle \nabla \rangle_{x'}^\sigma} (P_N)_{x'} \mathbb{E} \left[\Wick{\left( \langle \nabla \rangle ^\alpha Z_N   \right)^{k}}(t,x)\Wick{G(Z_M)}(s,x')   \right]\bigg|_{x'=x} \\
&= k! \langle \nabla \rangle_{x'}^{\alpha + n_0} \int^t_{-\infty} ds\ e^{(s-t) \langle \nabla \rangle_{x'}^\sigma} (P_N)_{x'} \prod_{i=1}^k \mathbb{E} \left[ \langle \nabla \rangle ^\alpha Z_N  (t,x)\langle \nabla \rangle^{n_i}Z(s,x')   \right]\bigg|_{x'=x}\\
&= k! \langle \nabla \rangle_{x'}^{\alpha + n_0} \int^t_{-\infty} ds\ e^{(s-t) \langle \nabla \rangle_{x'}^\sigma} (P_N)_{x'} \prod_{i=1}^k  \sum_{l\in\mathbb{Z}^d,\ |l|\le N} \langle l \rangle^{\alpha + n_i + 2m -\sigma} e^{(s-t)\langle l \rangle^\sigma}e_l(x) e_{-l} (x')    \bigg|_{x'=x} \\
&\simeq \sum_{l_i \in \mathbb{Z}^d_N} \frac{ \langle l_1 + \cdots + l_{k} \rangle^{\alpha + n_0}}{\sum_{i=1}^k \langle l_i \rangle^\sigma + \langle l_1 + \cdots + l_k \rangle^\sigma} \prod_{i=1}^k \langle l_i \rangle^{\alpha + n_i + 2m -\sigma}.
\end{align*}
Note now that for every $M$ with $2^M \le N$, there are of the order of $2^{kdM}$ values of 
$l$ such that $|l_i| \sim 2^M$ for every $i$ and $|\sum_i l_i| \sim 2^M$. It follows that 
\begin{equ}
c_{N,2}^\alpha
\gtrsim \sum_{2^M \le N}2^{kdM} 2^{M(\alpha + n_0-\sigma)} \prod_{i=1}^k 2^{M(\alpha + n_i + 2m-\sigma)}
= \sum_{2^M \le N} 2^{\delta(\alpha) M} \;,
\end{equ}
which proves \eqref{e:boundc2}.

Next, we prove the second assertion of the lemma. 
We write 
\[ X_N^\alpha \coloneqq \bigl\langle\Wick{\left( \langle \nabla \rangle ^\alpha Z_N   \right)^{k}}\langle \nabla \rangle ^\alpha Y_N  - c_{N,2}^\alpha, 1   \bigr\rangle_{L^2(\mathbb{T}^d)}\;, \]
and note that it is easy to see that
\begin{align}
\bigl\langle \Wick{\left( \langle \nabla \rangle ^\alpha Z_N   \right)^{k}}, e_l       \bigr\rangle_{L^2(\mathbb{T}^d)} &= \sum_{l_i\,:\, l_1 + \cdots + l_k = l}\Wick{\prod_{i=1}^k \one_{|l_i|\le N} \langle l_i \rangle^\alpha \hat Z (t,l_i)}\;, \label{saisho}
\end{align}
(here and below the $l_i$ take values in all of $\Z^d$, not just $\Z^d_N$) and 
\begin{align}
&\left\langle  \langle \nabla \rangle^\alpha Y_N(t)  , e_l       \right\rangle_{L^2(\mathbb{T}^d)}\notag \\
&= \Bigl\langle  \langle \nabla \rangle^\alpha P_N\langle \nabla \rangle^{n_0}\int_{-\infty}^t ds\ e^{(s-t)(1-\Delta)^{\sigma/2}}  \Wick{\prod_{i=1}^k\langle \nabla \rangle^{n_i}Z (s) }        , e_l       \Bigr\rangle_{L^2(\mathbb{T}^d)} \notag \\
&= \one_{|l| \le N} \sum_{l_i\,:\, l_1 + \cdots + l_k = l} \langle l \rangle^\alpha\langle l \rangle^{n_0}\int_{-\infty}^t ds e^{(s-t)(1+|l|^2)^{\sigma / 2}}\Wick{ \prod_{i=1}^k\langle l_i \rangle^{n_i}\hat Z(s,l_i)} . \label{nibann}
\end{align}
From \eqref{saisho} and \eqref{nibann}, we get
\begin{align*}
 X_N^\alpha (t) &=\sum_{(l,l')\in L_{N,d}}   \scal{\Sigma l}^{\alpha+n_0}\int_{-\infty}^t ds\ e^{(s-t)\scal{\Sigma l}^{\sigma}} \\
&\quad \times \biggl(\Wick{ \prod_{i=1}^k \langle l_i \rangle^\alpha \hat Z (t,l_i)}\Wick{ \prod_{i=1}^k\langle l'_i \rangle^{n_i}\hat Z(s, l'_i)} - \mathbb{E} \biggl[\Wick{ \prod_{i=1}^k \langle l_i \rangle^\alpha \hat Z (t,l_i)}\Wick{ \prod_{i=1}^k\langle l'_i \rangle^{n_i}\hat Z(s, l'_i)} \biggr] \biggr)\;,
\end{align*}
where $\Sigma l := \sum _{i=1}^k l_i$ for $l \in (\Z^d)^k$ and we write $L_{N,d}$ for the set of all $(l,l') \in (\Z^d)^k\times (\Z^d)^k$
such that 
\begin{equ}
|l_i| \le N\;,\qquad \bigl|\Sigma l'\bigr| \le N\;,\qquad \Sigma l' = -\Sigma l\;.
\end{equ}
From Lemmas~\ref{haii}--\ref{saigo} and the fact that the $k$ copies of $\scal{\nabla}^\alpha Z(t)$ are exchangeable, we get
\begin{equs}
X_N^\alpha (t) & =\sum_{(l,l')\in L_{N,d}}   \scal{\Sigma l}^{\alpha+n_0}\int_{-\infty}^t ds\ e^{(s-t)\scal{\Sigma l}^{\sigma}} \label{e:expressionXs}\\
&\times \sum_{S \Subset \left\{ 1, \cdots, k \right\}}  \frac{k!}{|S|!} \prod_{i\in S} \mathbb{E} \left[\langle l_i \rangle^\alpha \hat Z (t,l_i)\langle l'_i \rangle^{n_i}\hat Z(s,l'_i)\right]\Wick{\prod_{i\in S^c} \Bigl(\langle l_i \rangle^\alpha \hat Z (t,l_i) \langle l'_i \rangle^{n_i}\hat Z(s,l'_i)\Bigr)    }\;,
\end{equs}
where $S \Subset \left\{ 1, \cdots, k \right\}$ means that $S \subset \left\{ 1, \cdots, k \right\}$
but $S \neq \left\{ 1, \cdots, k \right\}$. The reason why the latter is ruled out is that the 
projection of $ X_N^\alpha (t)$ onto the $0$th Wiener chaos is cancelled out by the constant $c_{N,2}^\alpha$.
Since for $s \le t$ one has
\begin{equ}
\mathbb{E} \left[\hat Z (t,p)\hat Z(s,p')\right] = \scal{p}^{2m -\sigma}  \one_{p = -p'} e^{(s-t)\scal{p}^\sigma}\;,
\end{equ}
and writing $ X_N^\alpha (t) = \sum_{S \Subset \left\{ 1, \cdots, k \right\}}X_N^{\alpha, S}$
with $X_N^{\alpha, S}$ the corresponding term in \eqref{e:expressionXs} belonging to the chaos 
of order $q = 2(k - |S|)$, we note that $X_N^{\alpha, S}(t)$ is of the form
\begin{equ}
X_N^{\alpha, S}(t) = \int_{\CX}\Wick{\prod_{i \in S^c} f_i(x)} \,\mu(dx)\;,
\end{equ}
for suitable Gaussian random variables $f_i(x)$ and (signed) measure space $(\CX,\mu)$.
As a consequence of Corollary~\ref{contraction}, we then have
\begin{equ}
\mathbb{E} \left[ \bigl| X_N^{\alpha,S} (t) \bigr|^2    \right] 
\le q! \int_{\CX}\int_{\CX} \prod_{i \in S^c} \E\bigl( f_i(x)f_i(y)\bigr)\, \mu(dx)\,\mu(dy)\;,
\end{equ}
which, when written out explicitly, yields
\begin{equs}
\mathbb{E} \left[ \bigl| X_N^{\alpha,S} (t) \bigr|^2    \right]
&\lesssim \sum_{(l,l')\in L_{N,d}} \sum_{(L,L')\in L_{N,d}} \int_{-\infty}^t \int_{-\infty}^t ds du  \langle \Sigma l \rangle^{\alpha + n_0} \langle \Sigma L' \rangle^{\alpha + n_0} e^{(s-t)\langle \Sigma l \rangle^{\sigma}  + (u-t)\langle \Sigma L' \rangle^{\sigma}} \notag\\
&\quad \times \prod_{i\in S}\Bigl(\langle l_i \rangle^{\alpha + n_i +2m -\sigma}  \one_{l_i = -l'_{i}} e^{(s-t)\langle l_i \rangle^\sigma}\langle L_i' \rangle^{\alpha + n_i +2m -\sigma}  \one_{L_i = -L'_{i}} e^{(u-t)\langle L_i' \rangle^\sigma}\Bigr) \\
&\quad \times \prod_{i\in S^c}\Bigl(\scal{l_i}^{2\alpha +2m -\sigma}  \one_{l_i = -L_{i}} 
\langle L_i' \rangle^{2n_i +2m -\sigma}  \one_{l_i' = -L'_{i}} e^{-|u-s|\scal{L_i'}^\sigma} \Bigr)\;.  \label{qxn2} 
\end{equs}
Note first that the indicator functions completely determine $l'$ and $L$ once $l$ and $L'$ are given.
Furthermore, the constraints $\Sigma l' = -\Sigma l$ and $\Sigma L' = -\Sigma L$ enforce
the constraint $\sum_{i \in S^c} (l_i - L'_i) = 0$. A natural domain of summation is therefore
given by the set $L_{N,S}$ of pairs $(l,L)$ such that 
\begin{equ}
|l_i| \le N\;,\quad |L_j| \le N \,(\forall j \in S)\;,\quad |\Sigma L| \le N\;,\qquad \sum_{i \in S^c} (l_i - L_i) = 0\;.
\end{equ} 

Since $\int_{-\infty}^t \int_{-\infty}^t e^{-a(t-s)-b|u-s|-c(t-u)}\,ds\,du \lesssim a^{-\frac{2-\gamma}2}c^{-\frac{2-\gamma}2}b^{-\gamma}$ uniformly over $a,b,c > 0$ and $\gamma \in [0,1]$, we obtain for any $\gamma \in [0,1]$
\begin{equs}
\mathbb{E} \bigl[ \bigl| X_N^{\alpha,S} (t) \bigr|^2    \bigr] 
&\lesssim \sum_{(l,L)\in L_{N,S}}
 \langle \Sigma l \rangle^{\alpha + n_0} \langle \Sigma L \rangle^{\alpha + n_0}
 \prod_{i=1}^k\Bigl(\langle l_i \rangle^{\eta_i}  \langle L_i \rangle^{\beta_i} \Bigr) \\
 &\times \Bigl(\scal{\Sigma l}^\sigma + \sum_{i\in S}\scal{l_i}^\sigma\Bigr)^{-\frac{2-\gamma}2}
 \Bigl(\scal{\Sigma L}^\sigma + \sum_{i\in S}\scal{L_i}^\sigma\Bigr)^{-\frac{2-\gamma}2}\Bigl(\sum_{i\in S^c}\scal{L_i}^\sigma\Bigr)^{-\gamma} \;,
\end{equs}
where we set
\begin{equ}
\eta_i = 
\left\{\begin{array}{cl}
	\alpha + n_i + 2m-\sigma & \text{if $i \in S$,} \\
	2\alpha + 2m-\sigma & \text{otherwise,}
\end{array}\right.\qquad
\beta_i = 
\left\{\begin{array}{cl}
	\alpha + n_i + 2m-\sigma & \text{if $i \in S$,} \\
	2n_i + 2m-\sigma & \text{otherwise.}
\end{array}\right.
\end{equ}
Choose now any collection of $\sigma_i \ge 0$ such that $\sum_{i=0}^k \sigma_i = \sigma$ (note that the 
index starts at $0$, not $1$!) and set $\gamma = \sum_{i\in S^c}\sigma_i/\sigma$ in the above. This 
yields the bound
\begin{equ}[e:goodBoundX]
\mathbb{E} \bigl[ \bigl| X_N^{\alpha,S} (t) \bigr|^2    \bigr] 
\lesssim \sum_{(l,L)\in L_{N,S}}
 \langle \Sigma l \rangle^{\alpha + n_0 - \frac{\sigma+\sigma_0}2} \langle \Sigma L \rangle^{\alpha + n_0 - \frac{\sigma+\sigma_0}2}
 \prod_{i=1}^k\Bigl(\langle l_i \rangle^{\bar \eta_i}  \langle L_i \rangle^{\bar \beta_i} \Bigr)\;,
\end{equ}
where
\begin{equ}
\bar \eta_i = 
\left\{\begin{array}{cl}
	\alpha + n_i + 2m-\sigma - \frac{\sigma_i}2 & \text{if $i \in S$,} \\
	2\alpha + 2m-\sigma & \text{otherwise,}
\end{array}\right.\qquad
\bar \beta_i = 
\left\{\begin{array}{cl}
	\alpha + n_i + 2m-\sigma - \frac{\sigma_i}2 & \text{if $i \in S$,} \\
	2n_i + 2m-\sigma -\sigma_i& \text{otherwise.}
\end{array}\right.
\end{equ}

As a consequence of Assumption~\ref{w43}, we have
$2n_i +2m -\sigma > -d$ for all $i$ while, by Assumption~\ref{ass:subcritical}, we have 
$\sum_{i=1}^k (2n_i +2m -\sigma + d) = - 2A < d + \sigma$. We can therefore choose the
exponents $\sigma_i$ in such a way that
\begin{equ}[e:boundbetai]
2n_i + 2m-\sigma -\sigma_i > -d \quad (\forall i \ge 1)\;,\qquad \sum_{i=1}^k (2n_i + 2m-\sigma -\sigma_i+d) < d\;.
\end{equ}
The existence of such $\sigma _i$ is obtained by checking two cases: if $\sum_{i=1}^k (2n_i +2m -\sigma + d) \geq \sigma$ holds or not.
In particular, this implies that one also
has $\sum_{i \in S^c} (\bar \beta_i + d) < d$.

We are now precisely in the setup of Lemmas~\ref{lem:loop} and~\ref{lem:boundDivergent} 
with $\ell = 4$, as well as $k_1 = k_2 = |S^c|$ and $k_3 = k_4 = 1 + |S|$. 
Note that $(l,L) \in L_{N,S}$ implies that one does indeed have the identities
\begin{equ}
\sum_{i \in S^c} l_i = \sum_{i \in S^c} L_i = \Sigma l - \sum_{i \in S} l_i= \Sigma L - \sum_{i \in S} L_i\;.
\end{equ}
Furthermore, for any fixed enumerations
$o \colon |S| \to S$ and $\bar o \colon |S^c| \to S^c$, the exponents  $\alpha_{ij}$ are given by 
\begin{equs}[2]
\alpha_{1j} &= \bar \beta_{\bar o(j)}\;,&\qquad \alpha_{2j} &= \bar \eta_{\bar o(j)}\;,\quad  \quad j \le |S^c|\;,\\
\alpha_{4j} &= \bar \beta_{o(j)}\;,&\qquad \alpha_{3j} &= \bar \eta_{o(j)}\;,\quad  \quad j \le |S|\;,
\end{equs}
as well as
\begin{equ}[e:extravalues]
\alpha_{3,(|S|+1)} = \alpha_{4,(|S|+1)} = \alpha + n_0 - \frac{\sigma+\sigma_0}2\;.
\end{equ}
In particular, we have
\begin{equs}
\alpha_1 &= \sum_{i \in S^c} \bigl(2n_i +2m-\sigma - \sigma_i + d\bigr) - d\;,\\
\alpha_2 &= \sum_{i \in S^c} \bigl(2\alpha +2m-\sigma + d\bigr) - d\;,\\
\alpha_3 &= \alpha_4 = \alpha + n_0 -\frac{\sigma+\sigma_0}2 + \sum_{i \in S} \bigl(\alpha + n_i +2m-\sigma - \frac{\sigma_i}2+ d\bigr)\;,
\end{equs}
where the $\alpha_i$ are defined in Lemma~\ref{lem:loop}.

We first consider the case $\delta(\alpha) > 0$. In this case we infer from \eqref{e:boundbetai} that
the condition \eqref{e:assalphaij} holds for $i=1$ so that, 
for $\alpha$ sufficiently large, the assumptions of Lemma~\ref{lem:boundDivergent} are satisfied  
with $\bar \ell = 1$. This yields the bound 
\begin{equ}
\mathbb{E} \bigl[ \bigl| X_N^{\alpha,S} (t) \bigr|^2    \bigr] 
\lesssim N^{2\zeta}\;,
\end{equ}
with
\begin{equ}
\zeta = \alpha + n_0 - \sigma - \frac d2 + \sum_{i \le k} \bigl(\alpha + n_i + 2m-\sigma+d\bigr)
= \delta(\alpha) - \frac d2\;,
\end{equ}
so that the claim \eqref{e:boundZY} follows.

We now turn to the borderline case where $A = n_0-m- \frac\sigma2$ and $\delta(\alpha) = 0$, so that 
$\alpha = \alpha_0$.
Inserting this expression for $A$ into the definition \eqref{e:defalpha0} of
$\alpha_0$, we find that
\begin{equ}[e:expralpha0]
2\alpha_0 + 2m - \sigma =  \frac{d}{k+1} - d \;,
\end{equ}
so that, combining this again with a choice of $\sigma_i$ satisfying \eqref{e:boundbetai},
we find that  $\bar \eta_j > -d$ and $\bar \beta_j > -d$ for all
$j \in \{1,\ldots,k\}$. Inserting \eqref{e:expralpha0} into the expression for $\alpha_2$, we also find
that $\alpha_2 \le \frac{k}{k+1}d-d < 0$, so that  
assumption \eqref{e:assalphaij} of Lemma~\ref{lem:loop} is 
satisfied for $i=2$.

Combining the second inequality in \eqref{e:boundbetai} with the definition of $A$ in \eqref{e:defA},
we find that 
\begin{equ}[e:inequA]
2A + \sum_{i=1}^k \sigma_i > -d\;.
\end{equ}
Furthermore, using again \eqref{e:expralpha0}, we have
\begin{equs}
\alpha_{4,(|S|+1)} &= \alpha_0 + n_0 - \frac{\sigma+\sigma_0}2
= n_0 - m -\frac{d+\sigma_0}2 + \frac{d}{2(k+1)} \\
&= \frac12 \Bigl(2A + \sum_{i=1}^k \sigma_i + \frac{d}{k+1} - d\Bigr)
> -d + \frac{d}{2(k+1)}\;,
\end{equs}
where we made use of \eqref{e:inequA} in the last step. In particular, we have $\alpha_{ij} > -d$ 
for all $(i,j)$.

Since we know that $\sum \alpha_i = 2(\delta(\alpha_0)-d) = -2d < -d$, it remains to show that 
$\alpha_3 = \alpha_4 < 0$ in order to be able to apply Lemma~\ref{lem:loop}.   
Since $\alpha_0 + n_i +2m-\sigma - \frac{\sigma_i}2+ d > 0$ for every $i$ by 
\eqref{e:boundbetai} and the first inequality of \eqref{e:boundbetai}, we conclude that
\begin{equs}
\alpha_3 &< \alpha_0 + n_0 -\frac{\sigma+\sigma_0}2
+ \sum_{i =1}^k \bigl(\alpha_0 + n_i +2m-\sigma - \frac{\sigma_i}2+ d\bigr)\\
&=\alpha_0 + n_0 -\sigma
+ \sum_{i =1}^k \bigl(\alpha_0 + n_i +2m-\sigma+ d\bigr) = \delta(\alpha_0) = 0\;,
\end{equs}
as required.
\end{proof}

\appendix

\section{On the Besov spaces}\label{kpp}

We define the Besov spaces as follows. Let $\chi, \rho \in C_c ^\infty (\mathbb{R}^d)$ be $\mathbb{R}_+$-valued functions such that
\begin{enumerate}
\item
$\mathrm{supp}(\chi) \subset{B(4)},\  \mathrm{supp}(\rho) \subset{B(4)\backslash B(1)}$,
\item
$\chi (x) + \sum_{i=0}^\infty \rho (2^{-i}x) = 1$ for any $x \in \mathbb{R}^d$.
\end{enumerate}
Such a pair $(\chi,\rho)$ indeed exists, see \cite[Section 2.2]{BCD11}.
\ Then, writing
\[ \rho_{-1} = \chi, \ \rho_{j} = \rho(2^{-j}\cdot) \ \ \ \mathrm{for} \ j \ge 0, \]
we define the Littlewood-Paley blocks  by
\[ \Delta_m f(x) \coloneqq \frac{1}{(2\pi)^{\frac{d}{2}}} \sum_{l \in \mathbb{Z}^d} \rho_m(l) \hat{f}(l) e^{\sqrt{-1}l\cdot x} \ \ \mathrm{for} \ f \in \mathcal{D}'(\mathbb{T}^d).\]
The Besov spaces $B^\alpha_{p,q}(\mathbb{T}^d) $ for $1 \le p,q \le \infty, \alpha \in \mathbb{R}$ are then defined as the space of all $f \in \mathcal{D}'(\mathbb{T}^d)$ with the finite Besov norm
\[ \| f \|_{B^\alpha_{p,q} (\mathbb{T}^d)} \coloneqq \| 2^{m\alpha} \Delta_m f \|_{l_m^q(L^p (\mathbb{T}^d))} < \infty .\]
When $p=q=\infty$, we write $\mathcal{C}^\alpha (\mathbb{T}^d) \coloneqq B^\alpha_{\infty,\infty}(\mathbb{T}^d)$.
The proofs of the following two lemmas can be found in \cite[Lemma 2.1]{GIP15} and \cite[Lemma A.5]{GIP15}, respectively, for example.
\begin{lemm}[Product estimates]\label{pro}
For $\alpha, \beta \in \mathbb{R}$ with $\alpha + \beta >0$, 
\[ \|fg\|_{\mathcal{C}^{\alpha \wedge \beta}(\mathbb{T}^d)} \lesssim \|f\|_{\mathcal{C}^\alpha(\mathbb{T}^d)} \|g\|_{\mathcal{C}^\beta(\mathbb{T}^d)}. \]  
\end{lemm}

\begin{lemm}\label{kg}
Let $\alpha \in \mathbb{R}$ and $\gamma \ge 0$. Then, there exists some $C>0$ such that 
\[\|  N^{-\gamma} P_N f \|_{\mathcal{C}^{\alpha + \gamma} (\mathbb{T}^d)} \le C   \| f \|_{\mathcal{C}^{\alpha} (\mathbb{T}^d)}\]
for any $f\in \mathcal{C}^\alpha(\mathbb{T}^d)$ and $N \in \mathbb{N}$.
\end{lemm}
Combining the above two lemmas, we get the following lemma.

\begin{lemm}\label{vv}
Let $\gamma>0$, $\alpha_1, \cdots,  \alpha_K <0$ and $f_1 \in \mathcal{C}^{\alpha_1}(\mathbb{T}^d), \cdots, f_K \in \mathcal{C}^{\alpha_K}(\mathbb{T}^d)$. We assume that $\alpha \coloneqq \sum_{i=1}^{K}\alpha_i + \gamma >0$.
Then, there holds that 
\begin{align*}
\left\| N^{-\gamma} \prod_{k=1}^K P_N  f_k  \right\|_{\mathcal{C}^{\alpha} (\mathbb{T}^d)} \rightarrow 0
\end{align*}
as $N\rightarrow\infty$.
\end{lemm}

\section{On the Gaussian polynomials}\label{piyo}
We recall some basic facts on the Gaussian random variables, see \cite{Jan97} for more details, for example.
Let $H$ be a Gaussian space, namely a closed linear subspace of $L^2(\Omega, \P)$ whose elements all are centred, real-valued, Gaussian random variables.
The $n$th homogeneous Wiener chaos $H^{:n:} \subset L^2(\Omega, \P)$ is defined in the following way. We first define
$H^n$ as the closure in $L^2(\Omega,\mathbb{P})$ of the linear span of all products  
$f_1\cdots f_n$ with $f_i \in H \oplus \R$, so that in particular $H^{n-1} \subset H^{n}$,
and then define $H^{:n:}$ as the orthogonal complement of $H^{n-1}$ in $H^n$. (We set $H^{:0:} = H^0 = \R$.)
Letting $\mathcal{F}^H$ be the $\sigma$-algebra generated by $H$, it is known that
\[ L^2 (\Omega, \mathcal{F}^H, \mathbb{P}) = \bigoplus_{n=0}^\infty H^{:n:}.      \]
We denote by $\pi_n$ the orthogonal projection onto $H^{:n:}$.
For random variables $f_1, \cdots, f_k \in H$, we define their Wick product $\Wick{f_1 \cdots f_k}$ by 
\[\Wick{f_1 \cdots f_k} \coloneqq \pi_k (f_1 \cdots f_k) .\]

Regarding the Wick product, the following facts are known.
For $f\in H$ and $k\in\mathbb{N}_0$, one has
\begin{equation}  
\Wick{f^k} = H_k \left(f  ; \mathbb{E} [f^2]\right), 
\end{equation}
and more generally, for $f = (f_1, \cdots, f_k) \in H^k$ and $\beta \in \mathbb{N}_0^k$,
\begin{equation}\label{koio}
\Wick{f^\beta} = H_\beta \left(f; (\mathbb{E}[ f_i f_j])_{1\le i,j\le k} \right) ,     
 \end{equation}
 see Section~\ref{sec:notations} for the definition of $H_k$ and $H_\beta$.
The following lemmas are also known. The proofs can be found in \cite{Jan97}, for example.

\begin{lemm}\label{haii}
Let $f_1, \cdots, f_k, g_1, \cdots, g_k \in H$. Then, there holds
\begin{equ}[e:iso]
\mathbb{E} \left[\Wick{f_1 \cdots f_k}\Wick{g_1 \cdots g_k}   \right] = \sum_{\tau\in \mathfrak{S}_k}\prod_{i=1}^k \mathbb{E} \left[  f_i  g_{\tau(i)}  \right],
\end{equ}
where $\mathfrak{S}_k$ is the symmetric group of degree $k$. Especially, for $f, g \in H$, there holds
\begin{equ}[e:iso2]
\mathbb{E} \left[\Wick{f^{k}}\Wick{ g^{k}}   \right] = k! \mathbb{E} [fg]^k.
\end{equ}
\end{lemm}

\begin{cor}\label{contraction}
For any $n \ge 1$, the multilinear map $(f_1,\ldots,f_n) \mapsto \Wick{f_1 \cdots f_k}$
extends to a bounded linear map from $H^{\otimes n}$ to $L^2(\Omega,\P)$ with norm $\sqrt{n!}$.
\end{cor}

\begin{proof}
By \eqref{e:iso} and \eqref{e:iso2} the map factors through $H^{\otimes n} \twoheadrightarrow H^{\otimes_s n} \hookrightarrow L^2(\Omega,\P)$
with the first map the orthogonal projection and the second map an isometry (times $\sqrt{n!}$).
\end{proof}

\begin{lemm}\label{saigo}
Let $f_1, \cdots, f_k, g_1, \cdots, g_k \in H$. Then, there holds
\[\Wick{f_1 \cdots f_k}\Wick{g_1 \cdots g_k} = \sum_{S, S' \subset \left\{ 1, \cdots, k \right\}} \mathbb{E} \left[\Wick{\prod_{i\in S}f_i}\Wick{\prod_{i\in S'}g_i}\right] \Wick{\prod_{i\in S^c}f_i \prod_{i\in S'^c}g_i    }.    \]
In particular, $\Wick{f_1 \cdots f_k}\Wick{g_1 \cdots g_k} \in \bigoplus_{r=0}^k H^{:2r:}$ and
\[ \pi_{2r} \left(\Wick{f_1 \cdots f_k}\Wick{g_1 \cdots g_k}\right) = \sum_{S, S' \subset \left\{ 1, \cdots, k \right\}, |S| = |S'| = k-r} \mathbb{E} \left[\Wick{\prod_{i\in S}f_i}\Wick{\prod_{i\in S'}g_i}\right]\Wick{\prod_{i\in S^c}f_i \prod_{i\in S'^c}g_i    }.    \]

\end{lemm}

The following lemma is known as the hypercontractivity of Gaussian polynomials.
\begin{lemm}
Let $k \in \mathbb{N}$ and $p \in \mathbb{R}_+$. Then, there exists some constant $C_{p,k}$ such that for any $f \in H^{k} = \bigoplus_{n=0}^k H^{:n:}$, there holds 
\[  \mathbb{E} \left[ |f|^{2p}   \right] \le C_{p,k}  \mathbb{E} \left[  |f|^2  \right]^p       .\]

\end{lemm}

\section{On convergences of infinite sums}

The following is elementary

\begin{lemm}\label{exponents}
Let $\alpha_i > -d$ for $i=1,\ldots,k$ be such that
$\sum_{i=1}^k \alpha_i < -(k-1)d$. Then, for every non-empty subset $S \subset \{1,\ldots,k\}$,
one has $\sum_{i\in S} \alpha_i < -(|S|-1)d$.
\end{lemm}

\begin{proof}
The bound holds for $S = \{1,\ldots,k\}$ by assumption. Assume now that it holds for some non-empty $S$
and consider $S' = S \setminus \{j\}$ for some $j \in S$.
It follows that  $\sum_{i\in S'} \alpha_i = (\sum_{i\in S} \alpha_i) - \alpha_j
< -(|S|-1)d + d = -(|S'|-1)d$, and the claim follows.
\end{proof}

First, we recall the following basic estimates on discrete convolutions, see \cite[Lemma 2.3]{GKO24} for example.
\begin{lemm}\label{gre}
For $\alpha,\beta\in \R_-$ such that $\alpha,\beta> -d$ and $\alpha + \beta < -d$, the bound
\begin{equ}
\sum_{\ell \in \mathbb{Z}^d } \langle \ell \rangle^{\alpha} \langle k-\ell \rangle^{\beta}  \lesssim \langle k \rangle^{d + \alpha +  \beta  }\;,
\end{equ}
holds uniformly over $k \in \Z^d$.
\end{lemm}

The following result follows easily.
\begin{lemm}\label{lem:loop}
Let $\ell \ge 2$, let $k \in \N^\ell$ and set $K = \{(i,j) \in \N^2\,:\, i \le \ell\;\&\; j \le k_i\}$, let
$\alpha \in \R_-^K$ such that, for all $i$, one has
\begin{equ}[e:assalphaij]
\alpha_{ij} > -d\;,\qquad \sum_{j \le k_i} (\alpha_{ij} + d) < d\;,
\end{equ}
as well as
\begin{equ}[e:asssumalpha]
\sum_{i\le\ell} \alpha_i < -d\;,\quad   \alpha_i \eqdef {(k_i-1)d + \sum_{j \le k_i}\alpha_{ij}}\;.
\end{equ}
Write furthermore $\CP$ for the set of all elements $p \in (\Z^d)^K$ such that, for every $i,m \le \ell$,
one has $\sum_{j \le k_i} p_{ij} =  \sum_{j \le k_m}p_{mj}$.
Then, one has
\begin{equ}
\sum_{p \in \CP} \prod_{(i,j) \in K} \scal{p_{ij}}^{\alpha_{ij}} < \infty\;.
\end{equ}
\end{lemm}

\begin{proof}
Fix some $i$ and note first that, by induction on $k_i$, one easily deduces from
Lemma~\ref{gre} that, for any $q \in \Z^d$, one has 
\begin{equ}[e:boundpi]
\sum_{p \in L_q^{(i)}} \prod_{j \le k_i} \scal{p_{j}}^{\alpha_{ij}}  \lesssim \scal{q}^{\alpha_i}\;,    
\end{equ}
where we wrote $L_q^{(i)}$ for the set of $p \in (\Z^d)^{k_i}$ such that $\sum_{j \le k_i} p_{j} = q$ for every $j$. 
As a consequence, the expression in the lemma is bounded by $\sum_{q\in \Z^d} \scal{q}^{\sum_i \alpha_i}$.
Since $\sum_i\alpha_i < -d$ by assumption, the claim follows.
\end{proof}

For our purpose, it will be convenient to have a similar result, but with some of the summation variables
restricted to $|p_{ij}| \le N$. 

\begin{lemm}\label{lem:boundDivergent}
Let $\ell$, $k$, $K$, $\alpha$ be as in Lemma~\ref{lem:loop}, but we assume that there exists $\bar \ell < \ell$ such that  $\alpha_i < 0$ only for $i \le \bar \ell$, and we do not assume \eqref{e:asssumalpha}. 
For $i > \bar \ell$ on the other hand, we assume that $\alpha_{ij} > 0$ for every $j$.
Finally, we assume that $\sum_{i \le \bar \ell}\alpha_i > -d$.
Writing $\CP_N$ for the set of $p \in \CP$ such that $|p_{ij}| \le N$ for $i > \bar \ell$,
one then has
\begin{equ}[e:wantedBound]
\sum_{p \in \CP_N} \prod_{(i,j) \in K} \scal{p_{ij}}^{\alpha_{ij}} \lesssim N^{\gamma}\;,
\end{equ}
where $\gamma = d + \sum_{i\le \ell} \alpha_i$.
\end{lemm}

\begin{proof}
As in the proof of Lemma~\ref{lem:loop}, we fix some $q \in \Z^d$.
We then obtain again the bound \eqref{e:boundpi} for 
$i \le \bar \ell$, while for $i > \bar \ell$ we use the fact that
$\scal{p_{ij}} \lesssim N$ to conclude that 
\begin{equ}
\sum_{p \in L_{q,N}^{(i)}} \prod_{j \le k_i} \scal{p_{j}}^{\alpha_{ij}} \lesssim 
|L_{q,N}^{(i)}| \prod_{j \le k_i} N^{\alpha_{ij}} \lesssim N^{\alpha_{i}}\;,
\end{equ}
where $\alpha_i$ is as before and $L_{q,N}^{(i)} = \{p \in L_{q}^{(i)}\,:\, |p_j| \le N\;\forall j\}$.
As a consequence, the left-hand side of \eqref{e:wantedBound} is bounded by
\begin{equ}
\sum_{q \in \Z^d\,:\, |q| \lesssim N} \Bigl(\prod_{i \le \bar \ell}\scal{q}^{\alpha_i}\Bigr)\Bigl(\prod_{i > \bar \ell}N^{\alpha_i}\Bigr)\;, 
\end{equ}
from which the desired bound follows at once.
\end{proof}

\section*{Acknowledgements}
SK and HN thank Masato Hoshino for a fruitful discussion. MH would like to thank
Christophe Garban, Antti Kupiainen, and Jonathan Mattingly for inspiring discussions on this topic.
This work is supported by JSPS KAKENHI Grant Numbers 21H00988, 	23K20801, 24KJ1329, and JST SPRING Grant Number JPMJSP2110.

\bibliographystyle{Martin}
\bibliography{orthogonal.bib}

@book {BCD11,
    AUTHOR = {Bahouri, Hajer and Chemin, Jean-Yves and Danchin, Rapha\"{e}l},
     TITLE = {Fourier analysis and nonlinear partial differential equations},
    SERIES = {Grundlehren der mathematischen Wissenschaften [Fundamental
              Principles of Mathematical Sciences]},
    VOLUME = {343},
 PUBLISHER = {Springer, Heidelberg},
      YEAR = {2011},
     PAGES = {xvi+523},
      ISBN = {978-3-642-16829-1},
   MRCLASS = {35-02 (35L72 35Q30 42-02 42B37 76B03 76D03 76N10)},
  MRNUMBER = {2768550},
MRREVIEWER = {Peter\ R.\ Massopust},
       DOI = {10.1007/978-3-642-16830-7},
       URL = {https://doi.org/10.1007/978-3-642-16830-7},
}

@article{variance,
    AUTHOR = {Hairer, Martin},
     TITLE = {Renormalisation in the presence of variance blowup},
   JOURNAL = {Ann. Probab.},
  FJOURNAL = {The Annals of Probability},
    VOLUME = {53},
      YEAR = {2025},
      eprint={2401.10868},
    NUMBER = {5},
     PAGES = {1958--1985},
      ISSN = {0091-1798,2168-894X},
   MRCLASS = {99-06},
  MRNUMBER = {4962736},
       DOI = {10.1214/24-aop1756},
       URL = {https://doi.org/10.1214/24-aop1756},
}

@article {DPD03,
    AUTHOR = {Da Prato, Giuseppe and Debussche, Arnaud},
     TITLE = {Strong solutions to the stochastic quantization equations},
   JOURNAL = {Ann. Probab.},
  FJOURNAL = {The Annals of Probability},
    VOLUME = {31},
      YEAR = {2003},
    NUMBER = {4},
     PAGES = {1900--1916},
      ISSN = {0091-1798,2168-894X},
   MRCLASS = {81S20 (28C20 42B05)},
  MRNUMBER = {2016604},
       DOI = {10.1214/aop/1068646370},
       URL = {https://doi.org/10.1214/aop/1068646370},
}

@article {DPD02,
    AUTHOR = {Da Prato, Giuseppe and Debussche, Arnaud},
     TITLE = {Two-dimensional {N}avier-{S}tokes equations driven by a
              space-time white noise},
   JOURNAL = {J. Funct. Anal.},
  FJOURNAL = {Journal of Functional Analysis},
    VOLUME = {196},
      YEAR = {2002},
    NUMBER = {1},
     PAGES = {180--210},
      ISSN = {0022-1236,1096-0783},
   MRCLASS = {35Q30 (35R60 60H40 76D05 76M35)},
  MRNUMBER = {1941997},
MRREVIEWER = {Peter\ E.\ Kloeden},
       DOI = {10.1006/jfan.2002.3919},
       URL = {https://doi.org/10.1006/jfan.2002.3919},
}

@article {GIP15,
    AUTHOR = {Gubinelli, Massimiliano and Imkeller, Peter and Perkowski,
              Nicolas},
     TITLE = {Paracontrolled distributions and singular {PDE}s},
   JOURNAL = {Forum Math. Pi},
  FJOURNAL = {Forum of Mathematics. Pi},
    VOLUME = {3},
      YEAR = {2015},
     PAGES = {e6, 75},
      ISSN = {2050-5086},
   MRCLASS = {60H15 (35S50)},
  MRNUMBER = {3406823},
       DOI = {10.1017/fmp.2015.2},
       URL = {https://doi.org/10.1017/fmp.2015.2},
}

@article {GKO24,
    AUTHOR = {Gubinelli, Massimiliano and Koch, Herbert and Oh, Tadahiro},
     TITLE = {Paracontrolled approach to the three-dimensional stochastic
              nonlinear wave equation with quadratic nonlinearity},
   JOURNAL = {J. Eur. Math. Soc. (JEMS)},
  FJOURNAL = {Journal of the European Mathematical Society (JEMS)},
    VOLUME = {26},
      YEAR = {2024},
    NUMBER = {3},
     PAGES = {817--874},
      ISSN = {1435-9855,1435-9863},
   MRCLASS = {35L71 (35L81 35R60 60H15)},
  MRNUMBER = {4721025},
       DOI = {10.4171/jems/1294},
       URL = {https://doi.org/10.4171/jems/1294},
}

@article {Bjoern1,
    AUTHOR = {Bringmann, Bjoern},
     TITLE = {Invariant {G}ibbs measures for the three-dimensional wave
              equation with a {H}artree nonlinearity {I}: measures},
   JOURNAL = {Stoch. Partial Differ. Equ. Anal. Comput.},
  FJOURNAL = {Stochastics and Partial Differential Equations. Analysis and
              Computations},
    VOLUME = {10},
      YEAR = {2022},
    NUMBER = {1},
     PAGES = {1--89},
     eprint = {2009.04609},
      ISSN = {2194-0401,2194-041X},
   MRCLASS = {35R60},
  MRNUMBER = {4385403},
MRREVIEWER = {Le\ Chen},
       DOI = {10.1007/s40072-021-00193-y},
       URL = {https://doi.org/10.1007/s40072-021-00193-y},
}

@article {Bjoern2,
    AUTHOR = {Bringmann, Bjoern},
     TITLE = {Invariant {G}ibbs measures for the three-dimensional wave
              equation with a {H}artree nonlinearity {II}: dynamics},
   JOURNAL = {J. Eur. Math. Soc. (JEMS)},
  FJOURNAL = {Journal of the European Mathematical Society (JEMS)},
    VOLUME = {26},
      YEAR = {2024},
    NUMBER = {6},
    eprint = {2009.04616},
     PAGES = {1933--2089},
      ISSN = {1435-9855,1435-9863},
   MRCLASS = {35L15 (35R30 60H30)},
  MRNUMBER = {4742806},
       DOI = {10.4171/jems/1317},
       URL = {https://doi.org/10.4171/jems/1317},
}

@article {CQ,
    AUTHOR = {Coutin, Laure and Qian, Zhongmin},
     TITLE = {Stochastic analysis, rough path analysis and fractional
              {B}rownian motions},
   JOURNAL = {Probab. Theory Related Fields},
  FJOURNAL = {Probability Theory and Related Fields},
    VOLUME = {122},
      YEAR = {2002},
    NUMBER = {1},
     PAGES = {108--140},
      ISSN = {0178-8051,1432-2064},
   MRCLASS = {60G15 (60D05 60H05)},
  MRNUMBER = {1883719},
MRREVIEWER = {J.\ H\"usler},
       DOI = {10.1007/s004400100158},
       URL = {https://doi.org/10.1007/s004400100158},
}

@article {OOT,
    AUTHOR = {Oh, Tadahiro and Okamoto, Mamoru and Tolomeo, Leonardo},
     TITLE = {Focusing {$\Phi^4_3$}-model with a {H}artree-type
              nonlinearity},
   JOURNAL = {Mem. Amer. Math. Soc.},
  FJOURNAL = {Memoirs of the American Mathematical Society},
    VOLUME = {304},
      YEAR = {2024},
    NUMBER = {1529},
    eprint = {2009.03251},
     PAGES = {vi+143},
      ISSN = {0065-9266,1947-6221},
      ISBN = {978-1-4704-7193-4; 978-1-4704-8008-0},
   MRCLASS = {35L71 (35K15 35R60 60H15 60L40 81T08)},
  MRNUMBER = {4850409},
       DOI = {10.1090/memo/1529},
       URL = {https://doi.org/10.1090/memo/1529},
}

@article{liu2024,
      title={Existence, uniqueness, and universality of global dynamics for the fractional hyperbolic $\Phi^4_3$-model}, 
      author={Ruoyuan Liu and Nikolay Tzvetkov and Yuzhao Wang},
      year={2024},
      eprint={2311.00543},
      journal={arXiv preprints},
      primaryClass={math.AP},
      url={https://arxiv.org/abs/2311.00543}, 
}

@article {BG21,
    AUTHOR = {Barashkov, Nikolay and Gubinelli, Massimiliano},
     TITLE = {The {$\Phi^4_3$} measure via {G}irsanov's theorem},
   JOURNAL = {Electron. J. Probab.},
  FJOURNAL = {Electronic Journal of Probability},
    VOLUME = {26},
      YEAR = {2021},
     PAGES = {Paper No. 81, 29},
      ISSN = {1083-6489},
   MRCLASS = {81T08 (60H30 60L40)},
  MRNUMBER = {4269211},
MRREVIEWER = {Jiang\ Lun\ Wu},
       DOI = {10.1214/21-ejp635},
       URL = {https://doi.org/10.1214/21-ejp635},
}

@article {Hai14,
    AUTHOR = {Hairer, M.},
     TITLE = {A theory of regularity structures},
   JOURNAL = {Invent. Math.},
  FJOURNAL = {Inventiones Mathematicae},
    VOLUME = {198},
      YEAR = {2014},
    NUMBER = {2},
     PAGES = {269--504},
      ISSN = {0020-9910,1432-1297},
   MRCLASS = {60H15 (35R60 60H40 81S20 82C28)},
  MRNUMBER = {3274562},
MRREVIEWER = {Dora\ Sele\v{s}i},
       DOI = {10.1007/s00222-014-0505-4},
       URL = {https://doi.org/10.1007/s00222-014-0505-4},
}

@misc{HaiOrthogonal,
    AUTHOR = {Hairer, Martin},
	     TITLE = {{$\Phi^4_3$} is orthogonal to {GFF}},
	note ={Unpublished note},
	year={2017},
	url={https://hairer.org/Phi4.pdf}
}

@article {Konstantin,
    AUTHOR = {Hairer, M. and Matetski, K.},
     TITLE = {Discretisations of rough stochastic {PDE}s},
   JOURNAL = {Ann. Probab.},
  FJOURNAL = {The Annals of Probability},
    VOLUME = {46},
      YEAR = {2018},
    NUMBER = {3},
     PAGES = {1651--1709},
      ISSN = {0091-1798,2168-894X},
   MRCLASS = {60H15 (35R60 65M75)},
  MRNUMBER = {3785597},
MRREVIEWER = {Le\ Chen},
       DOI = {10.1214/17-AOP1212},
       URL = {https://doi.org/10.1214/17-AOP1212},
}

@article {Pavlos,
    AUTHOR = {Tsatsoulis, Pavlos and Weber, Hendrik},
     TITLE = {Spectral gap for the stochastic quantization equation on the
              2-dimensional torus},
   JOURNAL = {Ann. Inst. Henri Poincar\'e{} Probab. Stat.},
  FJOURNAL = {Annales de l'Institut Henri Poincar\'e{} Probabilit\'es et
              Statistiques},
    VOLUME = {54},
      YEAR = {2018},
    NUMBER = {3},
     PAGES = {1204--1249},
      ISSN = {0246-0203,1778-7017},
   MRCLASS = {60H15 (47A10 81T08)},
  MRNUMBER = {3825880},
MRREVIEWER = {Jiang\ Lun\ Wu},
       DOI = {10.1214/17-AIHP837},
       URL = {https://doi.org/10.1214/17-AIHP837},
}

@misc{Girsanov2,
    AUTHOR = {Watkins, Andrea},
	     TITLE = {On absolute continuity for stochastic partial differential equations
and an averaging principle for a queueing network},
	note ={PhD thesis, Duke University},
	year={2010},
	url={https://hdl.handle.net/10161/3048}
}

@article {Girsanov,
    AUTHOR = {Mattingly, Jonathan C. and Suidan, Toufic M.},
     TITLE = {The small scales of the stochastic {N}avier-{S}tokes equations
              under rough forcing},
   JOURNAL = {J. Stat. Phys.},
  FJOURNAL = {Journal of Statistical Physics},
    VOLUME = {118},
      YEAR = {2005},
    NUMBER = {1-2},
     PAGES = {343--364},
      ISSN = {0022-4715,1572-9613},
   MRCLASS = {76M35 (35Q30 35R60 60H15 76D05)},
  MRNUMBER = {2122959},
MRREVIEWER = {Hakima\ Bessaih},
       DOI = {10.1007/s10955-004-8787-3},
       URL = {https://doi.org/10.1007/s10955-004-8787-3},
}

@article {Glimm,
    AUTHOR = {Glimm, James},
     TITLE = {Boson fields with the {${:}\Phi \sp{4}{:}$} interaction in three
              dimensions},
   JOURNAL = {Comm. Math. Phys.},
  FJOURNAL = {Communications in Mathematical Physics},
    VOLUME = {10},
      YEAR = {1968},
     PAGES = {1--47},
      ISSN = {0010-3616,1432-0916},
   MRCLASS = {81.46},
  MRNUMBER = {231601},
MRREVIEWER = {W.\ G\"uttinger},
       URL = {http://projecteuclid.org/euclid.cmp/1103840981},
}

@article {Jeremy,
    AUTHOR = {Funaki, Tadahisa and Quastel, Jeremy},
     TITLE = {K{PZ} equation, its renormalization and invariant measures},
   JOURNAL = {Stoch. Partial Differ. Equ. Anal. Comput.},
  FJOURNAL = {Stochastic Partial Differential Equations. Analysis and
              Computations},
    VOLUME = {3},
      YEAR = {2015},
    NUMBER = {2},
     PAGES = {159--220},
      ISSN = {2194-0401,2194-041X},
   MRCLASS = {60H15 (35R60 82C28)},
  MRNUMBER = {3350451},
MRREVIEWER = {Yoshiki\ Otobe},
       DOI = {10.1007/s40072-015-0046-x},
       URL = {https://doi.org/10.1007/s40072-015-0046-x},
}

@book {Jan97,
    AUTHOR = {Janson, Svante},
     TITLE = {Gaussian {H}ilbert spaces},
    SERIES = {Cambridge Tracts in Mathematics},
    VOLUME = {129},
 PUBLISHER = {Cambridge University Press, Cambridge},
      YEAR = {1997},
     PAGES = {x+340},
      ISBN = {0-521-56128-0},
   MRCLASS = {60G35 (60H15 62M20 81T08)},
  MRNUMBER = {1474726},
MRREVIEWER = {Amarjit\ Budhiraja},
       DOI = {10.1017/CBO9780511526169},
       URL = {https://doi.org/10.1017/CBO9780511526169},
}

@article {MW17,
    AUTHOR = {Mourrat, Jean-Christophe and Weber, Hendrik},
     TITLE = {The dynamic {$\Phi^4_3$} model comes down from infinity},
   JOURNAL = {Comm. Math. Phys.},
  FJOURNAL = {Communications in Mathematical Physics},
    VOLUME = {356},
      YEAR = {2017},
    NUMBER = {3},
     PAGES = {673--753},
      ISSN = {0010-3616,1432-0916},
   MRCLASS = {60H15 (35K15 35K91 35R60)},
  MRNUMBER = {3719541},
       DOI = {10.1007/s00220-017-2997-4},
       URL = {https://doi.org/10.1007/s00220-017-2997-4},
}

@article {GH21,
    AUTHOR = {Gubinelli, Massimiliano and Hofmanov\'{a}, Martina},
     TITLE = {A {PDE} construction of the {E}uclidean {$\phi_3^4$} quantum
              field theory},
   JOURNAL = {Comm. Math. Phys.},
  FJOURNAL = {Communications in Mathematical Physics},
    VOLUME = {384},
      YEAR = {2021},
    NUMBER = {1},
     PAGES = {1--75},
      ISSN = {0010-3616,1432-0916},
   MRCLASS = {81S20 (35Q40 37D35 37E20 60G60 81T08 81T27 82B31)},
  MRNUMBER = {4252872},
MRREVIEWER = {Peter\ Bernard\ Weichman},
       DOI = {10.1007/s00220-021-04022-0},
       URL = {https://doi.org/10.1007/s00220-021-04022-0},
}

@article {GH19,
    AUTHOR = {Gubinelli, Massimiliano and Hofmanov\'{a}, Martina},
     TITLE = {Global solutions to elliptic and parabolic {$\Phi^4$} models
              in {E}uclidean space},
   JOURNAL = {Comm. Math. Phys.},
  FJOURNAL = {Communications in Mathematical Physics},
    VOLUME = {368},
      YEAR = {2019},
    NUMBER = {3},
     PAGES = {1201--1266},
      ISSN = {0010-3616,1432-0916},
   MRCLASS = {35Q40 (35J91 35K91 35R60 60H15 60H99 81T17 82C31)},
  MRNUMBER = {3951704},
MRREVIEWER = {Peter\ Bernard\ Weichman},
       DOI = {10.1007/s00220-019-03398-4},
       URL = {https://doi.org/10.1007/s00220-019-03398-4},
}

@article {AK20,
    AUTHOR = {Albeverio, Sergio and Kusuoka, Seiichiro},
     TITLE = {The invariant measure and the flow associated to the
              {$\Phi^4_3$}-quantum field model},
   JOURNAL = {Ann. Sc. Norm. Super. Pisa Cl. Sci. (5)},
  FJOURNAL = {Annali della Scuola Normale Superiore di Pisa. Classe di
              Scienze. Serie V},
    VOLUME = {20},
      YEAR = {2020},
    NUMBER = {4},
     PAGES = {1359--1427},
      ISSN = {0391-173X,2036-2145},
   MRCLASS = {81S20 (35Q40 37H10 60H15 60L50 81T08)},
  MRNUMBER = {4201185},
}

@article{LeonardoJames,
      title={Quasi-{G}aussianity of the {2D} stochastic {N}avier-{S}tokes equations}, 
      author={James Coe and Martin Hairer and Leonardo Tolomeo},
      year={2025},
      eprint={2510.13460},
      archivePrefix={arXiv},
      journal={arXiv preprints},
      primaryClass={math.PR},
      url={https://arxiv.org/abs/2510.13460}, 
}

@article {Barashkov,
    AUTHOR = {Barashkov, N. and Gubinelli, M.},
     TITLE = {A variational method for {$\Phi^4_3$}},
   JOURNAL = {Duke Math. J.},
  FJOURNAL = {Duke Mathematical Journal},
    VOLUME = {169},
      YEAR = {2020},
    NUMBER = {17},
     PAGES = {3339--3415},
      ISSN = {0012-7094,1547-7398},
   MRCLASS = {81T08 (60B10 60L40 93E20)},
  MRNUMBER = {4173157},
MRREVIEWER = {Jiang\ Lun\ Wu},
       DOI = {10.1215/00127094-2020-0029},
       URL = {https://doi.org/10.1215/00127094-2020-0029},
}

@article{AK22,
    AUTHOR = {Albeverio, Sergio and Kusuoka, Seiichiro},
     TITLE = {Construction of a non-{G}aussian and rotation-invariant
              {$\Phi^4$}-measure and associated flow on {$\Bbb R^3$} through
              stochastic quantization},
   JOURNAL = {Mem. Amer. Math. Soc.},
  FJOURNAL = {Memoirs of the American Mathematical Society},
    VOLUME = {308},
      YEAR = {2025},
    NUMBER = {1558},
     PAGES = {viii+114},
      ISSN = {0065-9266,1947-6221},
      ISBN = {978-1-4704-7318-1; 978-1-4704-8145-2},
   MRCLASS = {81S20 (35K58 35Q40 35R60 60H15 81T08)},
  MRNUMBER = {4896167},
       DOI = {10.1090/memo/1558},
       URL = {https://doi.org/10.1090/memo/1558}
}

@article {MRS22,
    AUTHOR = {Mattingly, Jonathan C. and Romito, Marco and Su, Langxuan},
     TITLE = {The {G}aussian structure of the singular stochastic {B}urgers
              equation},
   JOURNAL = {Forum Math. Sigma},
  FJOURNAL = {Forum of Mathematics. Sigma},
    VOLUME = {10},
      YEAR = {2022},
     PAGES = {Paper No. e75, 47},
      ISSN = {2050-5094},
   MRCLASS = {60H15 (60G15)},
  MRNUMBER = {4476105},
       DOI = {10.1017/fms.2022.64},
       URL = {https://doi.org/10.1017/fms.2022.64},
}

@article {BCCH21,
    AUTHOR = {Bruned, Y. and Chandra, A. and Chevyrev, I. and Hairer, M.},
     TITLE = {Renormalising {SPDE}s in regularity structures},
   JOURNAL = {J. Eur. Math. Soc. (JEMS)},
  FJOURNAL = {Journal of the European Mathematical Society (JEMS)},
    VOLUME = {23},
      YEAR = {2021},
    NUMBER = {3},
     PAGES = {869--947},
      ISSN = {1435-9855,1435-9863},
   MRCLASS = {60L30 (16T05 54C35)},
  MRNUMBER = {4210726},
MRREVIEWER = {Torstein\ K.\ Nilssen},
       DOI = {10.4171/jems/1025},
       URL = {https://doi.org/10.4171/jems/1025},
}

@article {Hai13,
    AUTHOR = {Hairer, Martin},
     TITLE = {Solving the {KPZ} equation},
   JOURNAL = {Ann. of Math. (2)},
  FJOURNAL = {Annals of Mathematics. Second Series},
    VOLUME = {178},
      YEAR = {2013},
    NUMBER = {2},
     PAGES = {559--664},
      ISSN = {0003-486X,1939-8980},
   MRCLASS = {35K59 (35B10 35B65 35R60 60G22 60H15 60K35)},
  MRNUMBER = {3071506},
MRREVIEWER = {Alp\ O.\ Eden},
       DOI = {10.4007/annals.2013.178.2.4},
       URL = {https://doi.org/10.4007/annals.2013.178.2.4},
}

@article {GP17,
    AUTHOR = {Gubinelli, Massimiliano and Perkowski, Nicolas},
     TITLE = {K{PZ} reloaded},
   JOURNAL = {Comm. Math. Phys.},
  FJOURNAL = {Communications in Mathematical Physics},
    VOLUME = {349},
      YEAR = {2017},
    NUMBER = {1},
     PAGES = {165--269},
      ISSN = {0010-3616,1432-0916},
   MRCLASS = {60H15 (58J65)},
  MRNUMBER = {3592748},
MRREVIEWER = {Peter\ Karl\ Friz},
       DOI = {10.1007/s00220-016-2788-3},
       URL = {https://doi.org/10.1007/s00220-016-2788-3},
}

@article {JP23,
    AUTHOR = {Jagannath, Aukosh and Perkowski, Nicolas},
     TITLE = {A simple construction of the dynamical {$\Phi^4_3$} model},
   JOURNAL = {Trans. Amer. Math. Soc.},
  FJOURNAL = {Transactions of the American Mathematical Society},
    VOLUME = {376},
      YEAR = {2023},
    NUMBER = {3},
     PAGES = {1507--1522},
      ISSN = {0002-9947,1088-6850},
   MRCLASS = {60H17 (35S50)},
  MRNUMBER = {4549683},
MRREVIEWER = {Jan\ I.\ Seidler},
       DOI = {10.1090/tran/8724},
       URL = {https://doi.org/10.1090/tran/8724},
}

@article{DGR24,
      title={Parabolic stochastic quantisation of the fractional {$\Phi^4_3$} model in the full subcritical regime}, 
      author={Paweł Duch and Massimiliano Gubinelli and Paolo Rinaldi},
      year={2024},
      journal={arXiv preprints},
      eprint={2303.18112},
      archivePrefix={arXiv},
      primaryClass={math.PR}
}

@article{Duc23,
    AUTHOR = {Duch, Pawe\l},
     TITLE = {Renormalization of singular elliptic stochastic {PDE}s using flow equation},
   JOURNAL = {Probab. Math. Phys.},
  FJOURNAL = {Probability and Mathematical Physics},
    VOLUME = {6},
      YEAR = {2025},
    NUMBER = {1},
     PAGES = {111--138},
      ISSN = {2690-0998,2690-1005},
   MRCLASS = {60H17 (81T17)},
  MRNUMBER = {4860691},
MRREVIEWER = {Jorge\ A.\ Le\'on},
       DOI = {10.2140/pmp.2025.6.111},
       URL = {https://doi.org/10.2140/pmp.2025.6.111},
}

@article{Duc22,
    AUTHOR = {Duch, Pawe\l},
     TITLE = {Flow equation approach to singular stochastic {PDE}s},
   JOURNAL = {Probab. Math. Phys.},
  FJOURNAL = {Probability and Mathematical Physics},
    VOLUME = {6},
      YEAR = {2025},
    NUMBER = {2},
     PAGES = {327--437},
      ISSN = {2690-0998,2690-1005},
   MRCLASS = {60H17 (81T17 82C28)},
  MRNUMBER = {4879385},
       DOI = {10.2140/pmp.2025.6.327},
       URL = {https://doi.org/10.2140/pmp.2025.6.327},
}

@article {CMW23,
    AUTHOR = {Chandra, Ajay and Moinat, Augustin and Weber, Hendrik},
     TITLE = {A priori bounds for the {$\Phi^4$} equation in the full
              sub-critical regime},
   JOURNAL = {Arch. Ration. Mech. Anal.},
  FJOURNAL = {Archive for Rational Mechanics and Analysis},
    VOLUME = {247},
      YEAR = {2023},
    NUMBER = {3},
     PAGES = {Paper No. 48, 76},
      ISSN = {0003-9527,1432-0673},
   MRCLASS = {60H17 (35Q40 60L30 81S20 82C28)},
  MRNUMBER = {4586861},
       DOI = {10.1007/s00205-023-01876-7},
       URL = {https://doi.org/10.1007/s00205-023-01876-7},
}

@article {AR91,
    AUTHOR = {Albeverio, S. and R\"ockner, M.},
     TITLE = {Stochastic differential equations in infinite dimensions:
              solutions via {D}irichlet forms},
   JOURNAL = {Probab. Theory Related Fields},
  FJOURNAL = {Probability Theory and Related Fields},
    VOLUME = {89},
      YEAR = {1991},
    NUMBER = {3},
     PAGES = {347--386},
      ISSN = {0178-8051,1432-2064},
   MRCLASS = {60H10 (34F05 34G99)},
  MRNUMBER = {1113223},
MRREVIEWER = {T.\ J.\ Lyons},
       DOI = {10.1007/BF01198791},
       URL = {https://doi.org/10.1007/BF01198791},
}

@article {CC18,
    AUTHOR = {Catellier, R\'emi and Chouk, Khalil},
     TITLE = {Paracontrolled distributions and the 3-dimensional stochastic
              quantization equation},
   JOURNAL = {Ann. Probab.},
  FJOURNAL = {The Annals of Probability},
    VOLUME = {46},
      YEAR = {2018},
    NUMBER = {5},
     PAGES = {2621--2679},
      ISSN = {0091-1798,2168-894X},
   MRCLASS = {60H15 (35K15 35R60 42B25 60H40)},
  MRNUMBER = {3846835},
MRREVIEWER = {Dora\ Sele\v si},
       DOI = {10.1214/17-AOP1235},
       URL = {https://doi.org/10.1214/17-AOP1235},
}

@article{Nag23,
    AUTHOR = {Nagoji, Hirotatsu},
     TITLE = {Construction of the {G}ibbs measures associated with
              {E}uclidean quantum field theory with various polynomial
              interactions in the {W}ick renormalizable regime},
   JOURNAL = {Electron. J. Probab.},
  FJOURNAL = {Electronic Journal of Probability},
    VOLUME = {29},
      YEAR = {2024},
     PAGES = {Paper No. 189, 42},
      ISSN = {1083-6489},
   MRCLASS = {81T08 (60H15)},
  MRNUMBER = {4841675},
       DOI = {10.1214/24-ejp1256},
       URL = {https://doi.org/10.1214/24-ejp1256},
}

@article {Nag24,
    AUTHOR = {Nagoji, Hirotatsu},
     TITLE = {Renormalization of stochastic nonlinear heat and wave
              equations driven by subordinate cylindrical {B}rownian noises},
   JOURNAL = {Stoch. Partial Differ. Equ. Anal. Comput.},
  FJOURNAL = {Stochastic Partial Differential Equations. Analysis and
              Computations},
    VOLUME = {12},
      YEAR = {2024},
    NUMBER = {2},
     PAGES = {932--967},
      ISSN = {2194-0401,2194-041X},
   MRCLASS = {60H15 (35K15 35L15 35R60)},
  MRNUMBER = {4734616},
       DOI = {10.1007/s40072-023-00296-8},
       URL = {https://doi.org/10.1007/s40072-023-00296-8},
}

@article {OOT24,
    AUTHOR = {Oh, Tadahiro and Okamoto, Mamoru and Tzvetkov, Nikolay},
     TITLE = {Uniqueness and non-uniqueness of the {G}aussian free field
              evolution under the two-dimensional {W}ick ordered cubic wave
              equation},
   JOURNAL = {Ann. Inst. Henri Poincar\'e{} Probab. Stat.},
  FJOURNAL = {Annales de l'Institut Henri Poincar\'e{} Probabilit\'es et
              Statistiques},
    VOLUME = {60},
      YEAR = {2024},
    NUMBER = {3},
     PAGES = {1684--1728},
      ISSN = {0246-0203,1778-7017},
   MRCLASS = {60H15 (37D35 60G15 60G60 60H30)},
  MRNUMBER = {4780502},
       DOI = {10.1214/23-aihp1380},
       URL = {https://doi.org/10.1214/23-aihp1380},
}


\end{document}